\documentclass{amsart}
\usepackage{amsmath,amssymb}
\usepackage{graphicx}
\newtheorem{theorem}{Theorem}[section]
\newtheorem{proposition}[theorem]{Proposition}
\newtheorem{corollary}[theorem]{Corollary}

\newtheorem{lemma}[theorem]{Lemma}

\theoremstyle{definition}
\newtheorem{definition}[theorem]{Definition}
\newtheorem*{assumption}{Assumption}

\theoremstyle{remark}
\newtheorem{remark}[theorem]{Remark}

\numberwithin{equation}{section}

\newcommand{\al}{\alpha}
\newcommand{\be}{\beta}
\newcommand{\de}{\delta}
\newcommand{\ep}{\varepsilon}

\newcommand{\ga}{\gamma}

\newcommand{\la}{\lambda}

\newcommand{\si}{\sigma}
\newcommand{\te}{\theta}
\newcommand{\vp}{\varphi}

\newcommand{\De}{\Delta}
\newcommand{\Ga}{\Gamma}
\newcommand{\La}{\Lambda}

\newcommand{\bD}{\mathbf{D}}
\newcommand{\bH}{\mathbf{H}}

\newcommand{\bLx}{\mathbf{L}^2_x}
\newcommand{\bL}{{\mathbf{L}^2}}

\newcommand{\tg}{{\widetilde g}}

\newcommand\hg{\widehat{g}}
\newcommand\Bg{\bar{g}}
\newcommand\Bga{\bar{\ga}}
\newcommand\Bh{\bar{h}}

\def\RR{\mathbb{R}}

\newcommand\bola{M}
\newcommand\obola{{\overline{M}}}
\newcommand\bdry{{\pd M}}

\newcommand{\ocA}{\overline{{\cA}}}

\newcommand{\cA}{{\mathcal A}}

\newcommand{\cC}{{\mathcal C}}
\newcommand{\cD}{{\mathcal D}}
\newcommand{\cE}{{\mathcal E}}
\newcommand{\tcE}{\widetilde{\mathcal E}}

\newcommand{\cF}{{\mathcal F}}
\newcommand{\tcF}{{\widetilde{\mathcal F}}}
\newcommand{\cG}{{\mathcal G}}
\newcommand{\cH}{{\mathcal H}}

\newcommand{\cL}{{\mathcal L}}
\newcommand{\cK}{{\mathcal K}}

\newcommand{\cO}{{\mathcal O}}
\newcommand{\cP}{{\mathcal P}}

\newcommand{\cS}{{\mathcal S}}

\newcommand{\cV}{{\mathcal V}}

\newcommand{\pd}{\partial}
\newcommand\minus\backslash

\newcommand{\e}{{\mathrm e}}

 \DeclareMathOperator\Ric{Ric}

\renewcommand\leq\leqslant
\renewcommand\geq\geqslant
\newcommand\less\lesssim
\newcommand\gr\grtsim

\newlength{\intwidth}

\newcommand\AdS{{\mathrm{AdS}_{n+1}}}

\newcommand\gref{\ga_0} 
\newcommand\bgref{\bar\ga_0}

\newcommand\poly{_{\mathrm{polyhom}}}

\newcommand\lot{\text{l.o.t.}}

\newcommand\pullback{(j_{(-T,T)\times\bdry})^{*}}

\newcommand{\triple}[1]{{\left\vert\kern-0.25ex\left\vert\kern-0.25ex\left\vert #1 
    \right\vert\kern-0.25ex\right\vert\kern-0.25ex\right\vert}}
\newcommand{\tripl}[1]{{\vert\kern-0.25ex\vert\kern-0.25ex\vert #1 
    \vert\kern-0.25ex\vert\kern-0.25ex\vert}}

\begin{document}

\title[Lorentzian Einstein metrics with prescribed conformal
infinity]{Lorentzian  Einstein metrics\\ with
  prescribed conformal infinity}

\author{Alberto Enciso}
\address{Instituto de Ciencias Matem\'aticas, Consejo Superior de
  Investigaciones Cient\'\i ficas, 28049 Madrid, Spain}
\email{aenciso@icmat.es}

\author{Niky Kamran}
\address{Department of Mathematics
  and Statistics, McGill University, Montr\'eal, Qu\'ebec, Canada H3A 2K6}
\email{nkamran@math.mcgill.ca}

%
%
\begin{abstract}
We prove a local well-posedness theorem for the $(n+1)$-dimensional Einstein equations 
in Lorentzian signature, with initial data $(\tg, K)$ whose asymptotic
geometry at infinity is similar to that anti-de Sitter (AdS) space,
and compatible boundary data $\hg$ prescribed at the time-like
conformal boundary of space-time. More precisely, we consider an
$n$-dimensional asymptotically hyperbolic Riemannian manifold
$(M,\tg)$ such that the conformally rescaled metric $x^2 \tg$ (with
$x$ a boundary defining function) extends to the closure $\obola$ of $M$ as a
metric of class $C^{n-1}(\obola)$ which is also polyhomogeneous of
class $C^{p}\poly(\obola)$. Likewise we assume that the conformally
rescaled symmetric $(0,2)$-tensor $x^{2}K$ extends to $\obola$ as a
tensor field of class $C^{n-1}(\obola)$ which is polyhomogeneous of class
$C^{p-1}\poly(\obola)$. We assume that the initial data $(\tg, K)$
satisfy the Einstein constraint equations and also that the boundary
datum is of class $C^p$ on $\partial M\times (-T_{0},T_{0})$ and
satisfies a set of natural compatibility conditions with the initial
data. We then prove that there exists an integer
$r_n$, depending only on the dimension $n$, such that if $p\geq
2q+r_n$, with $q$ a positive integer, then there is $T>0$, depending only on the norms of
the initial and boundary data, such that the Einstein
equations~\eqref{Einstein} has a unique (up to a diffeomorphism)
solution~$g$ on $(-T,T)\times\bola$ with the above initial and
boundary data,  which is such that $x^2g\in C^{n-1}((-T,T)\times\obola)\cap
C^q\poly((-T,T)\times\obola)$. Furthermore,
if $x^2\tg,x^2K$ are polyhomogeneous of class $C^\infty$ and  $\hg$ is
in $C^\infty((-T_0,T_0)\times\bdry)$, then $x^2g$ is in
$C^\infty\poly((-T,T)\times\obola)$.

\end{abstract}
\maketitle
\setcounter{tocdepth}{1}
\tableofcontents

\section{Introduction}
\label{S.intro}

Our goal in this paper is to prove a local well-posedness theorem for the $(n+1)$-dimensional Einstein equations 
\begin{equation}\label{Einstein}
\Ric(g)=-n g
\end{equation}
in Lorentzian signature, with initial data $(\tg, K)$ corresponding to
the asymptotic geometry of anti-de Sitter (AdS) space, and compatible
boundary data $\hg$ prescribed at the time-like conformal boundary of
space-time. More precisely, we consider an $n$-dimensional
asymptotically hyperbolic Riemannian manifold $(M,\tg)$, 
 such that the conformally rescaled metric $x^2 \tg$ extends to
 $\obola$, the union of $M$ with its boundary $\partial M$ (given by
 $x=0$), as a metric of class $C^{n-1}(\obola)$ which is polyhomogeneous of
 class $C^{p}\poly(\obola)$. Here and in what follows, $x$ is a
boundary defining function, that is a non-negative function on
$\obola$, smooth up to the boundary $\bdry$ of $M$, with $\bdry=\{x=0\}$ and such that the differential of
$x$ is nonzero on $\bdry$. We refer to Section~\ref{S.peeling} for the definition of polyhomogeneity.

Likewise we assume that the conformally rescaled symmetric
$(0,2)$-tensor $x^{2}K$ extends to $\obola$ as a tensor field of class
$C^{n-1}(\obola)$ which is polyhomogeneous of class
$C^{p-1}\poly(\obola)$. We
assume that the initial data $(\tg, K)$ satisfy the Einstein
constraint equations and also give boundary data of class $C^p$ on
$\partial M\times (-T_{0},T_{0})$ satisfying a set of natural
compatibility conditions with the initial data (we refer to
Appendix~\ref{A.constraint} for a discussion of the constraint
equations and compatibility conditions). The main result of our paper, which asserts that these initial and boundary data determine an Einstein
metric, can be stated as follows:

\begin{theorem}\label{T.main}
Suppose that we are given initial and boundary conditions $(\tg,K,\hg)$ with
$x^2\tg\in C^{n-1}(\obola)\cap C^p\poly(\obola)$, $x^2K\in
C^{n-1}(\obola)\cap C^{p-1}\poly(\obola)$ and $\hg\in
C^p((-T_0,T_0)\times \bdry)$ satisfying the constraint equations and
the compatibility conditions to order~$p$. There exists an integer
$r_n$, depending only on the dimension $n$, such that if $p\geq
2q+r_n$, then there is $T>0$, depending only on the norms of
the initial and boundary data, such that the Einstein
equations~\eqref{Einstein} has a unique (up to a diffeomorphism)
solution~$g$ on $(-T,T)\times\bola$ with the above initial and
boundary data,  which is such that $x^2g\in C^{n-1}((-T,T)\times\obola)\cap
C^q\poly((-T,T)\times\obola)$. Furthermore,
if $x^2\tg,x^2K\in
C^\infty\poly(\obola)$, $\hg\in C^\infty((-T_0,T_0)\times\bdry)$ and the
compatibility conditions are satisfied to all orders, then $x^2g\in C^\infty\poly((-T,T)\times\obola)$.
\end{theorem}


Hence the main result of our paper gives an extension to higher dimensions of the
fundamental pioneering work of Friedrich~\cite{Friedrich}, in which a general
existence theorem is proved for anti-de Sitter type space-times
in dimension $n+1=4$. The approach of~\cite{Friedrich} is based on a
reduction of the problem with boundary at infinity to a finite
maximally dissipative initial-boundary value problem, achieved through
an ingenious conformal representation of the Einstein equations in
dimension four. This leads to a general existence result for solutions
of the Einstein equations with negative cosmological constant
admitting a smooth conformal extension at space-like infinity. It is
should be noted that even though the results of~\cite{Friedrich} are
proved the assumption of smooth initial and boundary data, the method
used in~\cite{Friedrich} is flexible enough to allow for results on
metrics of lower regularity and can be extended to all even
space-time dimensions. 

The reason for which the method in~\cite{Friedrich} does not extend to the Einstein equations in  
odd space-time dimensions is that the metrics obtained through this approach are smooth
up to the boundary, while the Fefferman--Graham expansion~\cite{FG} implies that in odd dimension $n+1> 3$, the corresponding Einstein metric cannot have
this type of boundary regularity due to the appearance of log terms, which are present since the obstruction tensor does not vanish for a generic boundary datum in odd space-time dimensions. In the case of even
(e.g., four) space-time dimensions, this technical point has another subtle but significant effect: while the results
of~\cite{Friedrich} are finer than ours in the sense that initial data
that are smooth up the boundary are shown to yield Einstein metrics that are
also smooth up to the boundary (which is a stronger boundary regularity
result than the one we obtain), our result has the advantage that it also applies to
initial data that are only assumed to be polyhomogeneous, yielding polyhomogeneous
Einstein metrics. This is relevant because, even in four dimensions,
the solutions to the constraint equations constructed in~\cite{AC}
are generically polyhomogeneous (in fact, in $C^{n-1}(\obola)\cap
C^\infty\poly(\obola)$) but not smooth up to the boundary. (Notice however
that, despite this generic lack of smoothness up to the boundary, \cite{AC} does yield many nontrivial
solutions to the constraint equations that are smooth up to the
boundary and which give rise to many nontrivial Einstein metrics in
four dimensions directly using the breakthrough result of~\cite{Friedrich}.) Finally, it is worth
mentioning that we obtain explicit values for the constant~$r_n$ appearing in the statement of Theorem~\ref{T.main} (e.g., in four dimensions one can take
$r_3=17$) but that they are by no means sharp.

We shall see below that our purely PDE approach to
the formulation of the Einstein equations is of a different nature
from that of~\cite{Friedrich}, and that it uses instead as its
starting point some of the key similarities in the algebraic structure
of the Einstein equations between the cases of Lorentzian and
Euclidean signature. The existence of Einstein metrics in latter case
is well understood thanks to the work of Graham--Lee~\cite{GL},
Anderson~\cite{Anderson03,Anderson08}, Biquard~\cite{Biquard} and
others on the global existence and regularity of Riemannian Einstein
metrics with prescribed conformal infinity that are close, in a
suitable sense, to the hyperbolic metric. The situation in
Lorentzian signature is fundamentally different since it corresponds
to a hyperbolic evolution problem. Both the available analytical
techniques and the expected results are thus vastly different. In
particular, the metric $g$ is only guaranteed to exist  locally in
time (that is, for $|t|<T$), even for small data, a reflection of the
fact that the anti-de Sitter space is not expected to enjoy the good
stability properties of Minkowski space~\cite{CK} (we refer
to~\cite{B} and~\cite{F} for important recent work on the stability
problem for anti-de Sitter space). 

We would also like to mention that
besides the case of the Einstein equations considered
in~\cite{Friedrich}, the study of wave equations on asymptotically
anti-de Sitter spaces has attracted much attention in the last few
years. To the best of our knowledge, the wave equation on AdS$_4$ was
first considered by Breitenlohner and Freedman in~\cite{BF} using the
strong symmetry of the problem to separate variables. Again for
AdS$_4$, Choquet-Bruhat~\cite{Choquet1,Choquet2} proved global existence for the Yang--Mills equation under a radiation condition, and Ishibashi and Wald~\cite{IW} gave a proof of the well-posedness of the Cauchy problem for the Klein--Gordon equation in $\AdS$ using spectral theory. More refined results for the Klein--Gordon equation in an AdS space were developed by Bachelot~\cite{Bachelot1,Bachelot2,Bachelot3}, who used energy methods
and dispersive estimates to study the decay of the solutions and prove some results on the propagation of singularities.
In~\cite{Vasy}, Vasy established fine results on the propagation of singularities are proved for the Klein--Gordon equation on asymptotically AdS spaces using microlocal analysis. Holzegel and Warnick, both independently and in joint work~\cite{Holzegel,Warnick,HW}, used energy methods to prove the
well-posedness of the Cauchy problem for this equation in asymptotically AdS$_4$ space-times and discussed the boundedness of
solutions to the Klein--Gordon equation in stationary AdS black hole geometries. The local well-posedness for semilinear Klein--Gordon equations in asymptotically anti-de Sitter spaces with nontrivial boundary conditions at infinity was established
in~\cite{JMPA}. Spherically symmetric Einstein--Klein--Gordon systems have been considered in~\cite{WarnickNew}.

Finally, we mention that besides its interest as a question in geometric analysis and mathematical General Relativity, an important motivation for the problem of constructing Lorentzian Einstein manifolds with prescribed conformal infinity arises in the context of the AdS/CFT correspondence in string theory~\cite{Maldacena, Witten} (see~\cite{Anderson05,PRD} for further details on this point). The AdS/CFT correspondence is a conjectural relation which posits that a gravitational field on a Lorentzian $(n+1)$-manifold endowed with an asymptotically anti-de Sitter metric can be recovered from a conformal gauge field defined on the conformal boundary of the manifold. The gravitational field is typically modeled as a Lorentzian
metric $g$ satisfying the Einstein equation and the conformal gauge field corresponds to the conformal infinity $[\hg]$ of the metric. In this
setting, the {\em holographic principle}\/ asserts that the boundary data (which in the context of the Einstein equation would be the boundary metric
$\hg$), defined on the $n$-dimensional boundary, propagates through a suitable $(n+1)$-manifold referred to as the {\em bulk}\/ in the physics
literature) to determine the field (here the metric $g$) via a locally well-posed problem.

\section{Strategy of the proof}
\label{S.strategy}

In this section we will present the overall strategy of the proof of
Theorem~\ref{T.main}. We will also point out where the main points of
the argument can be found in the article, so this section also serves
as a guide to the paper.


Our first step is replace the Einstein equations~(\ref{Einstein}) with
modified Einstein equations taking the form of a quasilinear
hyperbolic system,  using what is often called DeTurck's
trick~\cite{DeTurck1,DeTurck2}. In the Riemannian case, this is amounts to writing the Einstein equations as an equivalent elliptic quasilinear system. 



The specific features of the quasilinear hyperbolic system
corresponding to Theorem~\ref{T.main} give rise to difficulties that make its proof rather involved, both technically and
conceptually. A first difficulty lies in the fact that asymptotically anti-de
Sitter metrics are not globally hyperbolic, so the classical local well-posedness result of Choquet-Bruhat~\cite{Ringstrom} does not apply. This is also reflected in
the fact that the modified Einstein equation when expressed in terms of the conformally rescaled metric ${\bar g}=x^{2}g$ contains terms 
that are strongly singular at the boundary $x=0$, so that the usual hyperbolic estimates are not enough
to control the behavior of the solutions of this equation. 
This requires the introduction of a functional framework adapted to
the geometry of these spaces. For this we rely on a scale of twisted,
weighted Sobolev spaces that are closely related to the spaces used in the edge
differential calculus~\cite{Mazzeo} but which we find more convenient
for our purposes.

A second difficulty is that in contrast to the globally hyperbolic case, where the modified Einstein equations correspond to a quasi-diagonal system (meaning that the leading part of the hyperbolic system 
is given by a scalar second-order differential operator, in our case the wave operator
$g^{\mu\nu}\pd_\mu\pd_\nu$), the leading part of the equations in the
asymptotically anti-de Sitter setting is no longer given by a quasi-diagonal
system. This is because the leading terms of the
equation (meaning the ones that cannot be absorbed into constants in
the estimates) are not only given by the second-order derivatives,
but also by additional terms that are singular at leading order when
$x=0$, and reflects the fact that, in the adapted
coordinates, the singularity at $x=0$ is critical from the point of
view of scalings. When these additional terms are taken into account, the
equation is no longer quasi-diagonal, so one must construct approximate
diagonalizations of the operators and take into account the fact that the 
estimates that we obtain in different ``eigenspaces'' are not
equivalent. It is remarkable, though, that the various powers of~$x$
that appear in scattered through the equations work together to allow us to prove Theorem~\ref{T.main}.

A third difficulty is that, in general, it is notoriously hard to impose
boundary conditions in the Einstein equations (see e.g.~\cite{FN} and
references therein). The way that we circumvent this problem is by
constructing the solution metric $g$ as a sum of two terms, one that
is ``large'' at infinity and which we construct using essentially
algebraic methods, and one which is ``small'' at infinity, whose existence must be proved using
analytic techniques, so that for all practical purposes one does not need to consider the boundary
conditions here.

Hence we are led to considering the following strategy in order to tackle the problem:
\medskip

\noindent{\em Step 1: The modified Einstein equation.} In
Section~\ref{S.modified} we discuss how one can replace the Einstein
equations~\eqref{Einstein} by a quasilinear hyperbolic system
$Q(g)=0$ using DeTurck's trick. Although from a conceptual point of view the argument
goes along familiar lines, the lack of global hyperbolicity makes technically
nontrivial some arguments needed to prove 
that both equations are equivalent. This is established in
Section~\ref{S.DeTurck} using ideas developed in the paper
(Theorem~\ref{T.DeTurck}).\smallskip

\noindent{\em Step 2: Peeling off the metric.} In
Section~\ref{S.peeling} we construct asymptotically anti-de Sitter
metrics $\ga_l$ that are ``approximate solutions'' to the modified
Einstein equation $Q(g)=0$ and satisfy the desired boundary conditions
(Theorem~\ref{T.peeling}). These metrics
have the property that $Q(\ga_l)$ is suitably small and are obtained
from the boundary datum~$\hg$ in an essentially algebraic
way that can be understood as peeling off the leading ``layers'' of the
solution at $x=0$, step by step. The parameter~$l$ corresponds to the
number of steps that one considers and is related to the norms in
which $\ga_l$ is an approximate solution of the modified Einstein
equations. One should notice that, in general, the rescaled metrics
$\Bga_l:=x^2\ga_l$ are not
smooth up to the boundary, but in some polyhomogeneous space $C^{n-1}\cap C^{p_l}\poly$.

\noindent{\em Step 3: Setting an iteration within a suitable
  functional framework.} To construct the metric~$g$ that solves the
modified Einstein equation, we write it as
\[
g=\ga + x^{\frac n2}u\,,
\]
where we have set $\ga:= \ga_l$ for a large enough~$l$. There
$\ga$ is going to be the ``large'' part at $x=0$ and the other terms
is going to be ``small'' at the boundary. 

To construct~$u$, we set up an iteration in
Section~\ref{S.iteration}. The convergence of this iteration will not
be proved until Section~\ref{S.convergence}, however. Before that, we need to
define suitable Sobolev spaces adapted to the geometry of the anti-de
Sitter space in which we can derive suitable estimates for~$u$. In
Sections~\ref{S.Sobolev} and~\ref{S.nonlinear} we consider two related
scales of Sobolev spaces, $\bH^{m,r}_\al$ and~$\cH^{m,r}$, and derive
several key estimates for them. It should be noticed that not only the
are proofs of these estimates different from those of the usual
Sobolev spaces $H^k(\RR^n)$, but so is also the case for the range of parameters
for which e.g.\ we have pointwise estimates (Corollary~\ref{C.Cmr}) or
can obtain estimates for the product of two functions
(Theorem~\ref{T.multilinear}).\smallskip

\noindent{\em Step 4: Linear estimates and convergence of the
  iteration.} Using the above adapted Sobolev spaces, in Section~\ref{S.linear} we obtain estimates for the
linear operators that appear in the iteration under certain assumptions
about the structure of the metric. Here the way that the various
powers of~$x$ appear is crucial to deriving the estimates that are
analogous (although the spaces and range of parameters are
different) to the usual ones obtained for globally hyperbolic quasilinear wave
equations. It should be emphasized though that the combination of the equation being effectively not quasi-diagonal with the fall-off of the
nonlinearities at the boundary make the analysis of the linear
equations and the treatment of the functional spaces much subtler than
in our previous paper~\cite{JMPA}, which was only concerned with scalar equations.

With these estimates in hand and equipped with the results about the
adapted Sobolev spaces established in the previous step, the proof of
the convergence of the iteration goes along the lines of the classical
result for globally hyperbolic spaces. The details are presented in
Section~\ref{S.convergence} although, as we have already mentioned, one has to wait until Theorem~\ref{T.DeTurck} to
show that these metrics are in fact Einstein.\medskip

The paper concludes with two appendices. In
Appendix~\ref{A.constraint} we recall the constraint and
compatibility conditions that must be imposed on the initial and
boundary data and the Andersson--Chrusciel result on the existence of
solutions to the constraint equations. In Appendix~\ref{A.AA} we record some results about the
integral operators $A_\al$ and~$A_\al^*$, defined in~\eqref{defAA*},
that we established in~\cite{JMPA}. These operators play an important
role in Sections~\ref{S.Sobolev} and~\ref{S.nonlinear}. For the benefit of the reader, we
also include a sketch of the proof.

\section{The modified Einstein equation}
\label{S.modified}

When dealing with the Einstein equation, a first difficulty, well understood by now,
is that the gauge invariance of the Einstein equation under changes of
coordinates makes it a very degenerate system. 
A standard way of solving this difficulty is using a technique that
is often called ``DeTurck's
trick''~\cite{DeTurck1,DeTurck2}, which employs a reference metric to
get rid of this gauge freedom. In the setting that we are considering,
it is important to choose a reference metric, which we will denote
by $\gref$, which a certain asymptotic behavior at infinity. To
avoid unnecessary repetitions, let us then begin by introducing the
following definition, where $I:=(-T_0,T_0)$ denotes a small interval of
the real line containing~$0$. 

\begin{definition}\label{D.waAdS}
A metric $g$ on $I\times\bola$ is called {\em weakly
  asymptotically AdS}\/ if the following conditions hold:
\begin{enumerate}
\item The rescaled reference metric $\Bg:=x^2g$ is of class
  $C^2$ up to the boundary.
\item The differential of the function~$x$ satisfies
  $\Bg^{\mu\nu}(\pd_\mu x)(\pd_\mu x)=1$ on $I\times\bdry$.
\end{enumerate}
\end{definition}

This definition is motivated by the formal calculations of Graham and
Lee in~\cite{GL}, many of which carry over verbatim to the case of
Lorentzian signature. The definition should be compared with that of an asymptotically AdS
metric, cf.~\cite{HT}.

We will choose the reference metric $\gref$ to be a weakly
asymptotically AdS metric on $I\times\bola$ such that the pullback of
\[
\bgref:=x^2\gref
\]
to $I\times\obola$ is $\hg$. 
A convenient way of doing this in terms
of the initial metric $g_0:=g|_{t=0}$, which we write in terms of the
initial data as described in Appendix~\ref{A.constraint}, is the following (we recall that the
pullback of $\Bg_0:=x^2 g_0$ to the boundary is precisely~$\hg$). Identifying $TI=I\times\RR$, for any $(t,z)\in I\times\bdry$
let us consider the tensor on
$T_{(t,z)}(I\times \bdry)=\RR\times T_z\bdry)$ given by
\[
G':=\hg|_{(t,z)}-\hg|_{(0,z)}\,.
\]
Now let $G$ be the only tensor on $T_{(t,z)}(I\times\bola)=\RR\times
T_z M$ which satisfies
\[
\pullback G=G'\,,\qquad
  (\Bg_0|_{z}+G)^{-1}dx=\Bg_0^{-1}|_z dx
\]
at $(t,z)$. Notice that, by continuity, the inverse appearing in the
second equation is well defined
provided that the interval $I$ is small enough. This defines a tensor field on $I\times\pd \bola$.

We can now extend $G$ to a tensor field $E(G)$ defined on a small neighborhood of
$I\times\bdry$, for instance by parallel transport with respect to
the metric $\bar g_0$ along integral curves of the gradient of~$x$. A suitable reference metric can then be constructed as $\gref:=x^{-2}\bgref$ with
\begin{equation}\label{ga0}
\bgref:= \Bg_0+\chi\, E(G)\,,
\end{equation}
with $\chi$ a suitable cutoff function that is equal to $1$ in a
neighborhood of the boundary. Notice that the reference
metric depends on the boundary and initial data and that it is a
(non-degenerate) Lorentzian metric because~$E(G)$ is
small if the interval is small.

Let us now denote by $\Ga^\nu_{\la\rho}$ and $\widetilde\Ga^\nu_{\la\rho}$
the Christoffel symbols of the metrics~$g$ and~$\gref$,
respectively. DeTurck's trick consists in looking for solutions to the
modified Einstein equation
\begin{equation}\label{eqQ}
Q(g)=0\,,
\end{equation}
where the components of the tensor $Q(g)$ are given in terms of those
of the Ricci tensor, $R_{\mu\nu}$, by
\begin{equation}\label{Q}
Q_{\mu\nu}:= R_{\mu\nu}+n g_{\mu\nu}+\frac12(\nabla_\mu W_\nu+\nabla_\nu W_\mu)\,.
\end{equation}
Here the covariant derivatives and the Ricci tensor are those of the
metric $g$ and the $1$-form $W$ is
\begin{equation}\label{W}
W_\mu:=g_{\mu\nu}\,g^{\la\rho}\,(\Ga^\nu_{\la\rho}-\widetilde\Ga^\nu_{\la\rho})\,.
\end{equation}
We will discuss the relationship between the solutions
of the Einstein equations~\eqref{Einstein} and those of the modified equation~\eqref{eqQ} in Section~\ref{S.DeTurck}, as the
lack of global hyperbolicity introduces some peculiarities. It is
worth mentioning that $Q(g)$ also depends on the initial and boundary conditions through the
reference metric~$\gref$.

It is well-known that the advantage of Equation~\eqref{eqQ} over the
Einstein equations is that the nondegeneracy has been taken care of;
indeed, \eqref{eqQ} is a quasilinear wave equation because
\begin{equation}\label{QDeg}
Q_{\mu\nu}=-\frac12 g^{\la\rho}\pd_\la\pd_\rho g_{\mu\nu}+
B_{\mu\nu}(g,\pd g)\,,
\end{equation}
with the second term quadratic in $\pd g$.
Our goal now is to solve the modified Einstein equation~\eqref{eqQ}
together with the compatible initial and
boundary conditions
\begin{align*}
g|_{t=0}=g_0\,,\qquad \pd_t g|_{t=0}=g_1\,,\qquad
\pullback \Bg=\hg\,.
\end{align*}
For the class of metrics that we
are considering, the coefficients are strongly singular at
$x=0$. Indeed, it essentially follows from a computation by Graham and
Lee~\cite[Equation~(2.19)]{GL} that for a weakly
asymptotically AdS metric $g$ one can
express~\eqref{QDeg} in terms of $\Bg$ as
\begin{equation}\label{Qsing}
Q_{\mu\nu}=\frac1{x^2}\,\bigg(n(1-\Bg^{\la\rho}x_\la
x_\rho)\,\Bg_{\mu\nu}-\frac12(B_\mu x_\nu+B_\nu x_\mu)\bigg) + \frac1x
\cP^1(\Bg)+\cP^2(\Bg)\,,
\end{equation}
where $x_\mu:=\pd_\mu x$,
\[
B_\mu:=\Bg^{\la\rho}(\Bga_0)_{\la\rho}\, \Bg_{\mu\nu}(\Bga_0)^{\nu\la}x_\la-(n+1)x_\mu\,,
\]
$\bgref:=x^2\gref$ and $\cP^1(\Bg)$ (respectively~$\cP^2(\Bg)$) stands for terms that depend
smoothly on $x$, $\Bg$, $\Bga_0$ and $\pd\Bga_0$ and are linear in
$\pd\Bg$ (respectively linear in $\pd^2\Bg$ and quadratic in $\pd\Bg$,
depending also on $\pd^2\Bga_0$).
Here all the indices are raised and lowered using the metric $\Bg_{\mu\nu}$ but
$(\Bga_0)^{\mu\nu}$, which is the inverse of $\Bga_0$.

In view of Equation~\eqref{Qsing}, we can immediately make the following
important observation:

\begin{proposition}\label{P.Qsing}
Suppose that $g$ is a weakly asymptotically AdS metric. Then $Q(g)=\cO(x^{-1})$
if and only if the following relations hold true on
$(-T,T)\times\bdry$:
\[
\Bg^{\mu\nu}(\bgref)_{\mu\nu}=n+1\qquad \text{and} \qquad \Bg^{\mu\nu}x_\nu=(\bgref)^{\mu\nu}x_\nu\,.
\]
\end{proposition}

\section{Peeling off the metric}
\label{S.peeling}

Throughout the defining function $x$ will be a $C^\infty$ positive
function on $\bola$ that vanishes to first order at the
boundary, which ensures that one can take it as a coordinate in a certain neighborhood of the boundary $\bdry$ in $ \bola$,
which we will denote by $\cA$. To parametrize $\cA$ we will always take coordinates $(x,\te)$,
where $\te=(\te^1,\dots,\te^{n-1})$ are local coordinates
on~$\bdry$. Since the analysis of the equation $Q(g)=0$ is only
problematic in a neighborhood of the boundary, these are the most convenient coordinates to carry
out the key estimates that are needed in this paper.

Let us start with some preliminary results that we will need to prove
the main result of this section. Here we denote by $\cS^2$ the space of symmetric covariant $2$-tensors
on $I\times\obola$. In the following proposition we provide a
convenient decomposition of this space at any point close to, or
lying on, the boundary $I\times\bdry$. {\em Throughout the section, we will assume that $g$ is a weakly
asymptotically AdS metric.}

\begin{proposition}\label{P.Vj}
In $I\times\ocA$, the space of symmetric tensors can be decomposed as
\[
\cS^2=\cV_0^{g}\oplus \cV_1^{g}\oplus \cV_2^{g}\oplus \cV_3^{g}\,,
\]
where
\begin{align*}
\cV_0^{g}&:=\big\{ H\in\cS^2: H_{\mu\nu}=\vp \Bg_{\mu\nu}\text{ with
  $\vp$ scalar} \big\}\,,\\
\cV_1^{g}&:=\big\{ H\in\cS^2: H_{\mu\nu}\Bg^{\nu\la}x_\la=0 \text{ and } H_{\mu\nu}\Bg^{\mu\nu}=0\}\,,\\
\cV_2^{g}&:=\big\{ H\in\cS^2: H_{\mu\nu}=\vp\,[(n+1)x_\mu
x_\nu-\Bg_{\mu\nu}] \text{ with $\vp$ scalar}\}\,,\\
\cV_3^{g}&:=\big\{ H\in\cS^2: H_{\mu\nu}=a_\mu x_\nu+a_\nu x_\mu
\text{ with } \Bg^{\la\rho}a_\la x_\rho=0\}\,.
\end{align*}
\end{proposition}
\begin{proof}
Since the $1$-form $dx$ does not vanish in $I\times\ocA$, it is easy to check
that $\cV_i^g\cap \cV_j^g=\{0\}$ if $i\neq j$ and that the dimensions of
the spaces $\cV_j^g$ at each point of $I\times \ocA$ are
\[
1\,,\quad \frac{n(n+1)}2-1\,,\quad 1\quad\text{and} \quad n\,,
\]
respectively. The sum of these numbers gives 
\[
\frac{(n+1)(n+2)}2\,,
\] 
that is, the dimension of $\cS^2$ at any point. The
proposition then follows.
\end{proof}

In what follows we will need more information about the structure of
the modified Einstein operator $Q(g)$ in a neighborhood of the
boundary. To analyze $Q(g)$, we will restrict our attention to the
set~$I\times \cA$ and use coordinates $(t,x,\te)$, where $\te$ are local coordinates on $\bdry$. It was computed by Graham and
Lee~\cite[Proposition~2.10]{GL} that the action of the differential of the map~\eqref{Q}
on a symmetric tensor $h=h_0+ h'$, with $h_0\in\cV_0^g$ and $h'\in
\cV_1^g\oplus \cV_2^g\oplus\cV_3^g$, is of the form
\begin{equation}\label{DQ}
(DQ)_g(h)=-\frac12\big( (\square_g-2n)h_0+(\square_g+2)h'\big)+ x\,\cL^1h\,,
\end{equation}
where $\square_g h_{\mu\nu}:=g^{\la\rho}\nabla_\la\nabla_\rho
h_{\mu\nu}$ is the wave operator on tensor fields and we henceforth use the notation $\cL^m$  for a matrix
$m^{\text{th}}$ order linear differential operator in the conormal derivatives
$(x\pd_x,\pd_\te,\pd_t)$ whose coefficients are smooth functions of
$(x,\Bg,\pd\Bg,\Bga_0,\pd \Bga_0,\pd^2 \Bga_0)$ up to $x=0$. In the
case $m=1$, the operator will not depend on $\pd^2\Bga_0$.

In particular, the part with second-order derivatives of the
linearized operator $(DQ)_g$ is the same as that of the wave operator
$-\frac 12\square_g$.  Regarding the terms that are most singular at
$x=0$, it was shown in~\cite[Proposition~2.7]{GL} that, in
terms of the coordinates $(t,x,\te)$, the Laplacian on a symmetric
tensor $h$ can be expanded in $x$ as
\begin{multline*}
\square_gh_{\mu\nu}= \big(
x^2\pd_x^2+(1-n)x\pd_x \big)h_{\mu\nu} +2 h_{\la\rho}\Bg^{\la\la}\Bg^{\rho\rho}x_\la x_\rho \Bg_{\mu\nu}\\
-(n+1)\big(h_{\mu\la}\Bg^{\la\rho}x_\rho x_\nu+
h_{\nu\la}\Bg^{\la\rho}x_\rho x_\mu\big) +2\Bg^{\la\rho}h_{\la\rho}x_\mu x_\nu
\\+
x\cL^1(h)_{\mu\nu}+ x^2\cL^2(h)_{\mu\nu}\,.
\end{multline*}

To further simplify this expression, let us define the quadratic polynomials
\[
p_j(s):=-\frac12\bigg(s-\frac n2+\al_j\bigg)\bigg(s-\frac n2-\al_j\bigg)\,,
\]
where $0\leq j\leq 3$ and $\al_j$ are the constants
\begin{equation}\label{alj}
\al_0:=\frac{\sqrt{n(n+8)}}2\,,\quad \al_1:=\frac n2\,,\quad
  \al_2:=\al_0\,,\quad \al_3:=\frac{\sqrt{n(n+4)}}2\,.
\end{equation}
In the following lemma, which we borrow from~\cite[Lemma 2.9]{GL} with a minor
change the notation, we use the subspaces $\cV_j^g$ to effectively
diagonalize $(DQ)_g$ up to terms that are smaller at $x=0$. Here
$p_j(x\pd_x)$ has the obvious meaning.

\begin{lemma}[\cite{GL}]\label{L.DQVj}
If $h\in \cV_j^g$, we have that
\[
(DQ)_g(h)_{\mu\nu}=x^{-2}\,p_j(x\pd_x) \Bh_{\mu\nu}+ x^{-1}(\cL^1 \Bh) _{\mu\nu}+ (\cL^2 \Bh)_{\mu\nu}\,.
\]
\end{lemma}

We will also need some information on the second derivative
$(D^2Q)_g$, understood as a quadratic form. For our purposes, it
will be enough to have the following symbolic description of
$(D^2Q)_g(h)$, where we are not displaying indices for the ease of notation:

\begin{lemma}\label{L.D2Q}
The second derivative of $Q$ is of the form
\[
(D^2Q)_g(h)=\cO(1) \Bh\,\pd^2\Bh +\cO(1)\, \pd \Bh\, \pd\Bh +
\cO(x^{-1})\, \Bh\, \pd\Bh+\cO(x^{-2})\, \Bh\, \Bh\,.
\]
Here we are using the notation $\Bh:=x^2 h$ and each term $\cO(x^{-s})$
above stands
for a smooth function of $x$, $\Bg$, $\pd\Bg$ and
the derivatives of $\Bga_0$ up to order $s$.
\end{lemma}

\begin{proof}
Ignoring the indices, we can use Eqs.~\eqref{QDeg}
and~\eqref{Qsing} to symbolically write the structure of $Q(g)$ as 
\[
Q(g)=\Bg^{-1}\pd^2 \Bg+a_0(\Bg)\, \pd\Bg\, \pd\Bg+ \frac{a_1(\Bg)}x\, \pd\Bg+\frac{a_2(\Bg)}{x^2}\,,
\]
where $a_j(\Bg)$ stands for a smooth function of $x$, $\Bg$ and the
derivatives of $\Bga_0$ up to order $j$. Since 
\[
Q(g+\ep h)=Q(g)+\ep \, (DQ)_g(h)+ \frac12\ep^2(D^2Q)_g(h) +\cO(\ep^3)\,,
\]
an elementary computation using that 
\[
(\Bg+\ep \Bh)^{-1}=\Bg^{-1}-\ep \Bg^{-1}\Bh\Bg^{-1}+\ep^2\Bg^{-1}\Bh\Bg^{-1}\Bh\Bg^{-1}+\cO(\ep^3)
\]
readily yields the desired expression for $(D^2Q)_g$.
\end{proof}

We will also need the following elementary fact:

\begin{lemma}\label{L.log}
For any integers $\si\geq0$ and $s$ there is a polynomial $f$ of
degree $\si$ or $\si+1$ such that 
\[
p_j(x\pd_x)\big(x^sf(\log
x)\big)=x^s(\log x)^\si\,.
\]
Furthermore, $f$ has degree $\si+1$ if and only if $p_j(s)=0$.
\end{lemma}

\begin{proof}
Since $p_j(0)\neq0$, it is clear that
\[
p_j(x\pd_x)\bigg(\frac1{p_j(0)}\bigg)=1\,.
\]
We now proceed by induction on $s$ and $\si$. Indeed, assume that the
statement holds true for all $s\leq s_0$ and $\si\leq \si_0$. The key
observation is that
\begin{multline}\label{formulalog}
p_j(x\pd_x)\big(x^s(\log x)^\si\big)=p_j(s) x^s(\log x)^\si-\frac {\si(2s+4-n)}2x^s(\log
x)^{\si-1}\\
-\frac{\si(\si-1)}2 x^s(\log x)^{\si-2}\,.
\end{multline}
If $p_j(s_0+1)\neq0$, by the induction hypothesis there is a
polynomial $F$ of degree at most $\si_0$ such that
\[
p_j(x\pd_x)\vp=x^{s_0+1}(\log x)^{\si_0}
\]
with
\[
\vp:=x^{s_0+1}\bigg(\frac{(\log x)^{\si_0}}{p_j(s_0+1)}+F(\log
x)\bigg) \,.
\]
On the other hand, if $p_j(s_0+1)=0$ we have that $s_0$ is $\frac n2\pm\al_j$, and in this case $2s_0+2-n$ is always nonzero. Hence
the induction hypothesis and the identity~\eqref{formulalog} ensure that
we can then take a polynomial $F$ of degree at most $\si_0$
such that 
\[
p_j(x\pd_x)\vp=x^{s_0+1}(\log x)^{\si_0}
\]
with
\[
\vp:=x^{s_0+1}\bigg(-\frac{2(\log x)^{\si_0+1}}{(\si_0+1)(2s_0+2-n)}+ F(\log
x)\bigg)\,.
\]
The same argument yields analogous functions $\psi$ with 
\[
p_j(x\pd_x)\psi=x^{s_0}(\log x)^{\si_0+1}
\]
and deals with the case of negative $s$, thereby completing the induction argument.
\end{proof}

Armed with these auxiliary results, we are now ready for the analysis
of the equation $Q(g)=0$ that we will carry out in
this section. For this we need to impose more stringent regularity assumptions on the
metric $\Bga_0$ than those in Section~\ref{S.modified}. Specifically, hereafter
we make the following regularity assumption:

\begin{assumption}[Regularity of the reference metric]
The metric $\Bga_0$ is of class $C^{n-1}\cap C^p\poly$ on $I\times\obola$.
\end{assumption}

Here $p\geq n-1$ a given integer and we recall that a function $h$ is
in $C^p\poly(I\cap\obola)$ (polyhomogeneous of class $C^p$) if it is of
class $C^p$ away from the boundary (say, on $I\times (M\backslash\cA)$) and a $C^p$ function of
$(t,x,\te,\log x)$ on a neighborhood of the
boundary, say $I\times\overline\cA$. The last condition means that in a small
neighborhood of
each point of $I\times\overline\cA$ there is a $C^p$ function $h'$ of
$n+2$ arguments such that
\[
h=h'(t,x,\te,\log x)\,.
\]
Since the pullback of the reference metric $\Bga_0$ to the boundary is
$\hg$, this regularity assumption implies that the boundary metric
$\hg$ must be of class $C^p(I\times\bdry)$.

To state the following theorem, we will introduce the space
$C^m_r(I\times\bola)$ of functions with $m+r$ continuous derivatives,
with the peculiarity that the last $r$ derivatives with respect to~$x$ are
regularized by multiplying by~$x$. This way, for instance, for all
$k,l,m\geq 1$ we have that
\begin{equation}\label{xlogxklm}
x^{m}\, (\log x)^l
\end{equation}
is in $C^{m-1}_k$ but not in $C^{m+k-1}$. To define the space $C^m_r(I\times\bola)$, we will
also use a
smooth nonnegative function $\chi_\cA$ of $x$ that
vanishes outside $I\times\cA$ and is equal to $1$ in a neighborhood of
$I\times\bdry$. With these objects at our disposal, we can now define
$C^m_r(I\times\bola)$ as the space of functions $\vp$
such that
\begin{equation}\label{Cmr}
\|\vp\|_{C^m_r(I\times\bola)}:=\|(1-\chi_\cA)\vp\|_{C^{m+r}(I\times\bola)}
 + \sum_{|\be|+j+k\leq r}\|(x\pd_x)^j\pd_t^k\pd_\te^\be(\chi_\cA\vp)\|_{C^m(I\times\bola)}
\end{equation}
is finite. The space $C^p_r(\bola)$ is defined analogously. (Of
course, the notation $\pd_\te^\be$ is somewhat heuristic as $\bdry$ is
not covered by a global chart. To define it rigorously, it is standard
that one can 
resort to either covering $\bdry$ with a fixed finite collection of
charts  and use a subordinate partition of unity, or to taking vector fields $X_1,\dots, X_M$ on $\bdry$ that span the whole
tangent space $T_p\bdry$ at each point $p\in\bdry$ and replace
$\pd_\te^\be\vp$ by
\[
X_1^{\be_1}\cdots X_M^{\be_M}\vp\,,
\]
with $|\be|=\be_1+\cdots +\be_M$. For notational simplicity, we will
stick to the notation $\pd_\te^\be$, which must be
interpreted in the aforementioned sense.)

We shall next present the main result of this section, which is a
procedure to obtain asymptotically anti-de Sitter metrics~$\ga$ that
satisfy the boundary condition $(j_{I\times\bdry})^* \Bga=\hg$ and
for which $Q(\ga)$ is suitably small. 
To state the theorem, we need to
introduce some notation. Given nonnegative integers $s$ and $\si$, we will say that a
symmetric tensor field $q$, of class $C^p$ in the interior of $I\times\bola$, is in $\cO_j(x^s\log^{\leq \si}x)$ if it can be
written in $\cA$ as
\[
q=x^s\sum_{\si'=0}^\si(\log x)^{\si'} B^{\si'}\,,
\]
where $B^{\si'}$ is a smooth symmetric tensor field in $I\times\obola$
satisfying the bounds
\[
\|B^{\si'}\|_{C^k(I\times\bola)}\leq
F_{k}\big(\|\Bga_0\|_{C^{k'}_{r'}(I\times\bola)}\big)
\]
for each $k\leq p-j$, where $k':=\min\{k+j,n-2\}$,
$r':=\max\{0,k+j-n+2\}$ and $F_{k}$ is a polynomial with $F_{k}(0)=0$. Although we will not
say it explicitly hereafter, it is important that in all the terms of
the form $\cO_j(x^s\log^{\leq \si} x)$ that will appear
in this section, the coefficients of the corresponding polynomials
$F_{k}$ will be uniformly bounded in terms of the $C^{n-1}\cap
C^p_{p-n+1}$ norm of $\Bga_0$.

\begin{theorem}\label{T.peeling}
Let us take a nonnegative integer $n-1\leq l\leq p$ and a 
small real $\de>0$. Then there is a weakly
asymptotically AdS metric $\ga_l$ on $I\times\bola$ of the form
\begin{equation*}
\ga_{l}=\sum_{k=0}^l \cO_k(x^{k-2}\log^{\leq \si_k}x)\,,
\end{equation*}
where each nonnegative integer $\si_k$ is zero for $k\leq n-1$,
such that:
\begin{enumerate}
\item The pullback to the boundary of $\Bga_l:=x^2\ga_l$ is
\[
(j_{I\times\bdry})^* \Bga_l=\hg\,.
\]

\item The metric $\ga_l$ is uniformly close to $\Bga_0$ in the sense
  that
\[
\|\Bga_l-\Bga_0\|_{L^\infty}<\de
\]
and furthermore
\begin{equation*}
\|\Bga_l\|_{C^{n-1}_{p-n+1}(I\times\bola)}<C
\end{equation*}
with a constant that depends only on $\|\Bga_0\|_{C^{n-1}_{p-n+1}}$
and~$\de$.

\item The metric $\ga_l$ is a solution of the modified Einstein
  equation almost to order $l-1$ in the sense that
\begin{equation*}
Q(\ga_l)=\cO_{l+1}(x^{l-1}\log^{\leq \si'_l}x)+ \cO_{l+2}(x^l\log^{\leq \si'_l}x)\,,
\end{equation*}
where $\si'_k$ is a nonnegative integer that is equal to zero for all
$k\leq n-1$. 
\end{enumerate}
\end{theorem}


\begin{proof}
Proposition~\ref{P.Qsing} trivially proves the result for
$l=0$. To see how things work for $l=1$, let us write the
$\cO_1(x^{-1})$ terms that appear in
\[
Q(\ga_0)=\cO_1(x^{-1})+ \cO_2(1)
\]
as
\[
\cO_1(x^{-1})=\frac{H_1}x+\cO_1(1)\,,
\]
where the tensor field $H_1$ is defined in terms of this quantity as
\begin{equation}\label{H1def}
H_1:=E\big(x \,\cO_1(x^{-1})|_{x=0}\big)
\end{equation}
and is $\cO_1(1)$. Here $E$ denotes the extension operator that we
introduced in Equation~\eqref{ga0}, and for the time being we will restrict
our attention to small values of~$x$.

Let us now use the direct sum decomposition of $\cS^2$ proved in
Proposition~\ref{P.Vj} to write in a unique way 
\[
H_1=\sum_{j=0}^3 H_{1j}\,,
\]
with $H_{1j}\in\cV_j^{\ga_0}$. We will take now
\[
\ga_1:=\ga_0-\sum_{j=0}^3 f_{1j}(x)H_{1j}
\]
with suitably chosen functions $f_{1j}(x)$. By Lemma~\ref{L.DQVj} and
Taylor's formula,
\begin{align*}
Q(\ga_1)&=Q(\ga_0)+(DQ)_{\ga}(\ga_1-\ga_0) +
I_1\\
& =x^{-2}\sum_{j=0}^3\bigg( x-p_j(x \pd_x)f_{1j}\bigg) H_{1j}
 + \cO_2(1)+(x\cL^1+ x^2\cL^2)(\ga_1-\ga_0)+I_1\,,
\end{align*}
where the error term is
\[
I_1:=\int_0^1(D^2Q)_{(1-s)\ga_0+s\ga_1}(\ga_1-\ga_0)\, ds\,.
\]
Since $p_j(-1)\neq0$, Lemma~\ref{L.log} ensures that we can take
functions $f_{1j}=\cO(x^{-1})$ (indeed, $f_{1j}(x)=x^{-1}/p_j(-1)$) such that
\[
p_j(x\pd_x)f_{1j}=x\,.
\]
Since $H_1$, in principle, is only defined in a neighborhood of the
boundary, we should include in $f_{1j}$ a suitable cut-off function,
which we henceforth omit for the ease of notation.
In any case, with this choice of $f_{1j}$ and Lemma~\ref{L.D2Q}, we obtain that
the error term is controlled by
\[
I_1=\cO_2(1)+\cO_3(x)\,,
\]
which immediately implies that
\[
Q(\ga_1)=\cO_2(1)+\cO_3(x)\,.
\]

The general case follows by an induction argument that also relies on
Taylor's formula and Lemmas~\ref{L.DQVj}--\ref{L.log}. To sketch the
proof, let us assume that the claim holds for all integers up to
$l-1$, with
\begin{equation*}
\Bga_{l-1}=\sum_{k=0}^{l-1}\cO_k(x^{k}\log^{\leq \si_k}x)
\end{equation*}
and $\si_k=0$ for all $k\leq n-2$.
To
prove it for $l$, we argue as above to write
\begin{equation}\label{refEq}
Q(\ga_{l-1})=x^{l-2}\sum_{j=0}^3\sum_{k=0}^{\si'_{l-1}}(\log x)^k H_{lkj}+ \cO_{l+1}(x^{l-1}\log^{\leq \si'_{l-1}}x)\,,
\end{equation}
with $H_{lkj}=\cO_{l}(1)$ a tensor field in
$\cV_j^{\ga_{l-1}}$ and $\si'_{l-1}$ an integer, related to
$\si_{l-1}$ and to the regularity of $\Bga_0$ up to the boundary,
which will be discussed later. Notice that $H_{lkj}$ can be assumed to be related to the
extension via the operator $E$ of a suitable tensor field defined on the
boundary, in an analogous fashion to~\eqref{H1def}. 

Lemma~\ref{L.log} allows us to take polynomials $f_{lkj}$, of
degree $k$ if $p_j(l)\neq0$ and $k+1$ otherwise, so that
\[
p_j(x\pd_x)\big(x^{l}f_{lkj}(\log x)\big)=x^{l}(\log x)^k\,.
\]
If we now set
\[
\Bga_{l}:=\Bga_{l-1}-x^{l}\sum_{j=0}^3\sum_{k=0}^{\si_{l-1}}f_{lkj}(\log x) H_{lkj}\,,
\]
a computation analogous to the one for $\ga_1$ then shows that 
\[
Q(\ga_{l}) =\cO_{l+1}(x^{l-1}\log^{\leq \si_{l}}x)+ \cO_{l+2}(x^{l}\log^{\leq \si_{l}}x)
\]
for some integer $\si_{l}$. 

Let us now complete our analysis of the log terms that appear in this
computation by discussing the values that $\si_{l-1}'$ can take. We have seen that $\si_l=0$ as long as $\si_k'=0$ for all
$k\leq l-1$ and $p_j(l)\neq0$.  That is, log terms appear in $\ga_l$ either
through log terms the right hand side of Equation~\eqref{refEq} (where
they can come from log terms in $\Bga_{l-1}$ or from the reference
metric $\Bga_0$, which is in $C^{n-1}\cap C^p\poly$ and therefore such that its
first non-smooth term is of the form  $x^n\log x$) or due to the existence of
integer roots of a polynomial $p_j(s)$, as shown in
Lemma~\ref{L.log}. It follows from Equation~\eqref{alj} that the first
integer root of a polynomial $p_j(s)$ is $p_1(n)=0$, so log terms can
only appear at order $x^n\log x$ in $\Bga_l$ and we therefore get that
$\Bga_l$ is of class $C^{n-1}\cap C^p\poly$.

Since $\Bga_l-\Bga_0$ vanishes at $x=0$, it suffices to take the
support of the aforementioned cut-off functions to be small enough to
ensure that $\|\Bga_l-\Bga_0\|_{L^\infty}$ is as small as one wishes.
Besides, it is apparent from the construction that the tensor fields $H_{lkj}$ that
appear at the $l^{\text{th}}$ step of the induction that the
coefficients are bounded in terms of $\Bga_0$ and its
$l^{\text{th}}$ order derivatives
which yields the estimate 
\[
\|\Bga_l\|_{C^{n-1}_{p-n+1}}<C\,,
\]
with $C$ a constant that depends on $\|\Bga_0\|_{C^{n-1}_{p-n+1}}$.
Of course, the reason for
which in general we get this estimate in $C^{n-1}_{p-n+1}$ but not
in $C^p$ is the presence of log terms in the expression for $\Bga_l$ 
starting with $x^n\log x$. 
\end{proof}

\section{Setting the iteration}
\label{S.iteration}

Our goal in this section is to set up an iterative procedure that will
eventually lead to a solution of the equation $Q(g)=0$ with the
desired initial and boundary conditions. To this end, let us write the
solution as
\[
g=:\ga+ h\,,
\]
where
\[
\ga:= \ga_l
\]
is the metric constructed in Theorem~\ref{T.peeling} with some large
enough value of the parameter~$l$ that we will specify later. We will
also assume that the number~$p$ appearing the regularity assumption of
Section~\ref{S.peeling} is large enough. Intuitively, the weakly asymptotically AdS metric~$\ga$ is
the part of the metric that is ``large'' at the boundary and $h$ is ``smaller''. 

Let us recall from Equation~\eqref{QDeg} that one can write $Q(g)$ in local
coordinates as
\[
Q(g)=\widetilde P_gg+B(g)\,,
\]
where we define the $g$-dependent linear differential
operator~$\widetilde P_g$ as
\[
(\widetilde P_gg')_{\mu\nu}:=-\frac12 g^{\la\rho}\pd_\la\pd_\rho g_{\mu\nu}'
\]
and $B(g)$ depends on $g$ and quadratically on $\pd g$. Taylor's
formula ensures that
\begin{equation}\label{Qg1}
B(g)=B(\ga)+(DB)_\ga h-\tcE(h)\,,
\end{equation}
where the error term is
\begin{equation}\label{tcE}
\tcE(h):=-\int_0^1(D^2B)_{\ga+sh}(h)\, ds
\end{equation}
and the second order differential of $B$ is understood as a quadratic
form. The equation $Q(g)=0$ can then be written as
\begin{equation}\label{Qg2}
\widetilde P_gh+(DB)_gh+(\widetilde P_\ga\ga-\widetilde P_g\ga)+Q(\ga)-\tcE(h)=0\,.
\end{equation}

Let us now define a linear operator, depending on $g$, as
\[
T_gh:=-3h(\nabla^{(\ga)}x, \nabla^{(\ga)}x)\, \Bg\,,
\]
where $\nabla^{(\ga)}$ stands for the connection associated with the
metric~$\ga$. As easy computation shows that $T_g$ is the
differential of the function $g\mapsto \widetilde P_\ga\ga-\widetilde P_g\ga$ at $g=\ga$. Hence we will set
\begin{equation}\label{tcF}
\tcF(h):=T_gh+\widetilde P_g\ga-\widetilde P_\ga\ga\,,
\end{equation}
which, in view of~\eqref{Qg2}, allows us to write the equation
$Q(g)=0$ as
\[
\widetilde P_gh+(DB)_gh+T_gh=-Q(\ga)+\tcF(h)+\tcE(h)\,.
\]
Let us now define another $g$-dependent linear differential operator $P_g$
by setting
\[
\widetilde P_gh+(DB)_gh+T_gh=:x^{\frac n2+2}P_gu\,,
\]
where we have introduced the new unknown $u$ as
\[
h=:x^{\frac n2}u\,.
\]
Full details about the structure of the differential operator will be
given in Section~\ref{S.linear}.  In terms of~$u$, the equation
$Q(g)=0$ can be finally written as
\begin{equation}\label{EinsteinP}
P_gu=\cF_0+\cG(u)\,,
\end{equation}
where
\[
\cG(u):=\cF(u)+\cE(u)
\]
and
\[
\cF_0:=-x^{-\frac n2-2}Q(\ga)\,,\qquad \cF(u):=x^{-\frac
  n2-2}\tcF(h)\,,\qquad \cE(u):=x^{-\frac n2-2}\tcE(h)\,.
\]

In the forthcoming sections our objective will be to solve this equation using an iterative
procedure that will produce $u$ as the limit of a sequence $u^m$,
with $u^1:=0$ and 
\begin{equation*}
P_{g^{m}}u^{m+1}=\cF_0+\cG(u^{m})\,.
\end{equation*}
Of course, here $g^m:= \ga+ x^{\frac n2}u^m$ and the initial conditions
that we need to impose are
\begin{equation*}
u^{m+1}|_{t=0}=u_0\,,\qquad \pd_t u^{m+1}|_{t=0}=u_1\,,
\end{equation*}
where we have set
\begin{equation}\label{u01}
u_j:=x^{-\frac n2}(g_j-\pd_t^j\ga|_{t=0})
\end{equation}
for each nonnegative integer~$j$, with $g_j:=\pd_t^jg|_{t=0}$. As we will see, the compatibility conditions of the
initial and boundary data boil down to assumption that a certain
number of the functions $u_j$ fall off fast enough at $x=0$ to be in a
suitable space of square-integrable functions over~$M$. Since $g_j$ is
just a time derivative of the metric at $t=0$, and therefore
determined by the initial datum of the problem (that is, a Riemannian
metric on $M$ and a second fundamental form satisfying the constraint
equations), and $\ga$ was determined by algebraically solving the
Einstein equations to a certain order near the boundary, this just
means that the formal series expansions for the solution that we get
from the initial and boundary data must be compatible to a certain
order. In the terminology of~\cite{Friedrich}, this is means imposing corner
conditions to a finite order.

\section{Adapted Sobolev spaces}
\label{S.Sobolev}

In this section we will introduce some twisted Sobolev spaces that are
adapted to the AdS geometry near the conformal boundary. They will be
key in our derivation of the estimates that will allow us to prove the
convergence of the iteration presented in the previous
section. Specifically, we will consider two kinds of adapted Sobolev spaces,
$\bH^m_\al$ and $\cH^m$, as well as certain modifications of them,
$\bH^{m,r}_\al$ and $\cH^{m,r}$, that play a role somewhat similar to that of
the spaces $C^m_r$ introduced in~\eqref{Cmr}. The first kind of
adapted spaces depends on a parameter~$\al$ that in our applications
will ultimately be one of the quantities~$\al_j$ defined
in~\eqref{alj}, so {\em we will assume throughout that $\al>1$}\/ without
further mention. The
properties of these spaces for $\al<1$ are quite different, as
discussed in~\cite{JMPA}.

To define the spaces $\bH^m_\al$, let us begin by introducing the twisted derivative with parameter $\al$ as
\[
\bD_{x,\al}\vp :=\pd_x\vp+\frac\al x\vp \,.
\]
Its formal adjoint in the Hilbert space
\begin{equation}\label{bLx}
\bLx:=L^2((0,\infty),x\, dx)
\end{equation}
is 
\[
\bD^*_{x,\al}\vp:=-\pd_x\vp +\frac{\al-1} x\vp \,,
\]
and we will set
\begin{equation}\label{D(m)}
\bD_{x,\al}^{(k)}\vp:=\begin{cases}
(\bD_{x,\al}^*\bD_{x,\al})^\frac k2\vp & \text{if } k \text{ is even},\\[1mm]
\bD_{x,\al} (\bD_{x,\al}^*\bD_{x,\al})^{\frac{k-1}2}\vp & \text{if } k \text{ is odd},
\end{cases}
\end{equation}
with the proviso that $\bD_{x,\al}^{(0)}\vp:=\vp$.

The twisted Sobolev space $\bH^m_\al\equiv \bH^m_\al(\bola)$ is defined as
follows. Let us suppose that the function~$u$ is supported in a small
neighborhood of the boundary~$\bdry$, which we will take as
\[
\cA:=\big\{(x,\te)\in(0,a)\times\bdry\big\}\,.
\]
We can then define its $\bH^m_\al$
norm as
\begin{equation*}
\|u\|_{\bH^m_{\al}(\cA)}^2:=\sum_{j+|\be|\leq m}\int_{\bdry}\int_0^{a}
|\bD_{x,\al}^{(j)}\pd_\te^\be u|^2\, x\, dx\, d\te
\end{equation*}
where $d\te$ is the canonical measure on the sphere and the twisted
derivative acts on~$u$ in the obvious way. Using a suitable
cutoff function that is equal to $1$ in a neighborhood of the boundary
and vanishes outside $\cA$, for a function $u$ defined on the ball
we can then set
\begin{equation}\label{bH}
\|u\|_{\bH^m_\al}:=\|\chi u\|_{\bH^m_{\al}(\cA)}+ \|(1-\chi)u\|_{H^m(\bola)}\,,
\end{equation}
where $H^m$ is the usual Sobolev space. The space
$\bH^m_\al$ can then be defined as the closure in this norm of the space of smooth
functions on~$M$ of compact support, the definition being also
applicable to tensor-valued functions using standard arguments.  For
$m=0$ the norm, which does not depend on~$\al$, will be simply
denoted by $\|u\|_{\bL}$ or occasionally by $\|u\|$.

For real $s>0$, we can use interpolation to define the space
$\bH^s_\al\equiv \bH^s_\al(\bola)$. Equivalently, since $\bD_{x,\al}^*\bD_{x,\al}$ is
an essentially self-adjoint operator in $L^2(\RR^+,x\, dx)$ with
the domain $C^\infty_0(\RR^+)$, we can write
\begin{equation}\label{Hkball}
\|u\|_{\bH^s_\al}:=\big\|\La_\al^s(\chi
u)\big\|_{\bL} + \|(1-\chi)u\|_{H^s(\bola)}\,,
\end{equation}
where
\begin{equation}\label{Laal}
\La_\al^s:=(1-\De_{\bdry}+\bD_{x,\al}^*\bD_{x,\al})^{s/2}
\end{equation}
is defined using the spectral theorem.  As we did in~\eqref{Cmr}, we can also consider the space with~$m$
derivatives as above and~$r$
``regularized'' derivatives. For this we use
the norm that is defined as
\[
\|u\|_{\bH^{m,r}_\al}:= \sum_{j+|\be|\leq r}\|(x\,\pd_x)^j\pd_\te^\be
(\chi u)\|_{\bH^m_\al} + \|(1-\chi)u\|_{H^{m+r}(\bola)}\,.
\]

Closely related scales of Sobolev spaces are $\cH^m\equiv
\cH^m(\bola)$ and $\cH^{m,r}\equiv\cH^{m,r}(\bola)$, which do not
depend on any parameters and are weighted variations of the spaces typically
considered in the theory of differential edge operators (see
e.g.~\cite{Mazzeo}). They are respectively defined as the closure of
$C^\infty_0(\bola)$ in the norm
\begin{align*}
\|u\|_{\cH^m}&:= \sum_{j+|\be|\leq m}
\|x^{j-m}\pd_x^j\pd_\te^\be (\chi u)\|_{\bL} + \|(1-\chi)u\|_{H^{m+r}(\bola)}\,,\\
\|u\|_{\cH^{m,r}}&:= \sum_{j=0}^r\|x^j
u\|_{\cH^{m+j}}\,,
\end{align*}
in each
case. Notice that these norms are constructed by including in each
derivative a singular weight that depends on the number of
$x$-derivatives that one is taking. These spaces can also
be defined for non-integer values using interpolation or, denoting by
$\pd_x^*:=-\pd_x-1/x$ the formal adjoint of~$\pd_x$ with respect to
the $\bLx$ product,
directly through the formula
\begin{align}\label{cHs}
\|u\|_{\cH^{s}}:=\bigg\|\bigg(\frac{1-\De_{\bdry}}{x^2}+\pd_x^* \pd_x\bigg)^{s/2}(\chi
u)\bigg\|_{\bL} + \|(1-\chi)u\|_{H^s(\bola)} \,.
\end{align}
In particular, this ensures that the
usual interpolation formulas are valid for these scales of Sobolev
spaces.

We shall need estimates relating the various adapted Sobolev spaces
that we have introduced. A simple observation is the following, which
show how multiplication by powers of~$x$ can help us redistribute
the ``standard'' and ``regularized'' derivatives in the
spaces~$\bH^{m,r}_\al$ and~$\cH^{m,r}$:

\begin{proposition}\label{P.power}
Given nonnegative integers $m$, $r$ and an integer
$l\in[-m,r]$, we have the inequality
\[
\|x^lu\|_{\cH^{m,r}}\leq C \|u\|_{\cH^{m+l,r-l}}\,.
\]
\end{proposition}
\begin{proof}
It is enough to expand the various terms appearing in the definitions
of the norm and use some elementary algebra.
\end{proof}

To explore the properties of these spaces we will make use of the
integral operators
\begin{subequations}\label{defAA*}
\begin{align}
A_\al\vp(x)&:=x^{-\al}\int_0^x y^\al \vp(y)\, dy\,,\\
A_\al^*\vp(x)&:=x^{\al-1}\int_x^1 y^{1-\al}\vp(y)\, dy\,,
\end{align}
\end{subequations}
which act on functions of one variable and will play an essential role
in the rest of this section. Notice that these operators are right inverses of $\bD_{x,\al}$ and $\bD_{x,\al}^*$ in the
sense that
\[
\bD_{x,\al}(A_\al\vp)=\bD_{x,\al}^*(A_\al^*\vp)=\vp\,;
\]
in particular, $A_\al^*$ is the adjoint of $A_\al$ in
$\bLx$. Obviously $A_\al,A_\al^*$ also act on functions defined on
$\cA$. In Appendix~\ref{A.AA} we record some important properties of
these operators, extracted from~\cite{JMPA}.

A simple but important estimate is the following, which gives an $L^\infty$
bound for functions belonging to an adapted Sobolev space. Notice that,
contrary to what happens in the usual Sobolev embedding theorem, we
are not asking for the square-integrability of $\frac n2+\ep$ derivatives
but actually of $\frac{n+1}2+\ep$:

\begin{theorem}\label{T.Morrey}
Let $u\in\bH^{1,r}_{\al}$ with $r>\frac {n-1}2$. Then 
we have the pointwise estimate in the ball
\[
\|u\|_{L^\infty}\leq C\, \|u\|_{\bH^{1,r}_\al} \,.
\]
\end{theorem}
\begin{proof}
By the definition of the norm and the Sobolev embedding, it is
obviously enough to prove the result for~$u$ supported in~$\ocA$. But for a.e.~$(x,\te)$ in
$\cA$ we then have
\begin{align*}
|u(x,\te)|&=\big|A_\al^*(\bD_{x,\al}u)(x,\te)|\\
&\leq C \|\bD_{x,\al}u(\cdot,\te)\|_{\bLx}\\
&\leq C \|\bD_{x,\al} u \|_{\bLx H^r_\te}\\
&\leq C\|u\|_{\bH^{1,r}_\al}\,,
\end{align*}
where $H^r_\te\equiv H^r(\bdry)$ is the Sobolev space of functions on $\bdry$
with $r$ square-integrable derivatives and to pass to the first, second and third lines we have
respectively used the properties~(i) and~(iii) in Theorem~\ref{T.AA*}
and the Sobolev embedding. The theorem then follows.
\end{proof}

\begin{corollary}\label{C.Cmr}
For any $\rho>\frac {n-1}2$, $\|u\|_{C^m_r(\bola)}\leq C \|u\|_{\cH^{m+1,r+\rho}}$.
Furthermore, we have the bound
\[
\|x^{-m}(x\,\pd_x)^j \pd_\te^\be u\|_{L^\infty(\cA)}\leq C\|u\|_{\cH^{m+1,r+\rho}}
\]
for all indices with $j+|\be|\leq m+r$.
\end{corollary}
\begin{proof}
It stems Theorem~\ref{T.Morrey} and the
fact that $x^{-m}(x\,\pd_x)^j \pd_\te^\be u\in \cH^{1,\rho}$
for the above range of indices whenever $u\in\cH^{m+1,r+\rho}$.
\end{proof}

The connection between the spaces~$\bH^{m,r}_\al$ and~$\cH^{m,r}$ is
subtler. Of course, the estimate
\begin{equation}\label{relationSobolev}
\|u\|_{\bH^{m,r}_\al}\leq C\|u\|_{\cH^{m,r}}
\end{equation}
follows from an elementary computation. That for some range of the
parameters there is a converse to this inequality, so that the
norms~$\bH^{m,r}_\al$ and~$\cH^{m,r}$ are equivalent, is more
sophisticated.  The following theorem is the partial converse to the
inequality~\eqref{relationSobolev} that we need:

\begin{theorem}\label{T.bH}
For any $k\leq m$, if $\al>k-1$,
\[
\|u\|_{\cH^{k,r+m-k}}\leq C\|u\|_{\bH^{m,r}_\al}\,.
\]
In particular, both norms are equivalent if $\al>k-1$.
\end{theorem}

\begin{proof}
Since $u\in \bH^{m,r}_\al$ if and only if $(x\,
\pd_x)^j\pd_\te^\be u\in\bH^m_\al$ for all $j+|\be|\leq r$, it is clearly enough to prove that
\[
\|u\|_{\cH^{k,m-k}}\leq C\|u\|_{\bH^{m}_\al}
\]
whenever $\al>k-1$. There is no loss of generality in proving the
result for functions supported in~$\cA$, since away from the boundary
both norms are equivalent.

With $m=1$, it suffices to see that one can write
\[
u=A_\al(\bD_{x,\al}u)
\]
as a consequence of Theorem~\ref{T.AA*} and that, due to this theorem,
\[
\bigg\|\frac ux\bigg\|_{\bL}\leq C\|\bD_{x,\al}u\|_{\bL}\leq C\|u\|_{\bH^1_\al}\,.
\]
Hence
\[
\|\pd_xu\|_{\bL}=\bigg\|\bD_{x,\al}u-\al\frac ux\bigg\|_{\bL}\leq
\|\bD_{x,\al}u\|_{\bL}+\al\bigg\|\frac ux\bigg\|_{\bL}\leq  C\|u\|_{\bH^1_\al}\,,
\]
as we wanted to prove.

Let us now consider the case $m=2$. A moment's thought reveals that it
is enough to keep track of derivatives with respect to $x$ in the argument, which is
what we will do here, because we have that $\pd_\te^\be u\in \bH_\al^{m-|\be|}$.
Hence let us start by using Theorem~\ref{T.AA*} to write
\[
\bD_{x,\al}u=A^*_\al(\bD_{x,\al}^{(2)}u)+ x^{\al-1}f_1(\te)\,,
\]
where $f_1(\te)$ is a function on the sphere satisfying
$\|f_1\|_{L^2_\te}\leq C\|u\|_{\bH^2_\al}$. Here we are using the
notation $L^2_\te\equiv L^2(\bdry)$. Again by Theorem~\ref{T.AA*}, this implies
\[
u=A_\al^{(2)}(\bD_{x,\al}^{(2)}u)+cx^\al f_1(\te)\,,
\]
where $c$ is a constant and we are using the notation
\[
A_{\al}^{(l)}\vp:=\begin{cases}
(A_\al^*A_\al)^{\frac l2}\vp & \text{if } l \text{ is even},\\[1mm]
A_\al(A_\al^*A_\al)^{\frac{l-1}2}\vp & \text{if } l \text{ is odd}.
\end{cases}
\]
The desired estimates follow from this formula and the properties of
the operators $A_\al$ and $A_\al^*$ listed in Theorem~\ref{T.AA*}. In
order to see this, we start by noticing that
\begin{align}
\bigg\|\frac{A_\al^{(2)}\vp}{x^2}\bigg\|_{\bL}&=\bigg\|\frac
1{x^{\al+2}}\int_0^xy^\al A_\al^*\vp(y)\, dy\bigg\|_{\bL}\notag\\
&=\bigg\|\frac 1xA_{\al+1}\bigg(\frac{A^*_\al\vp}x\bigg)\bigg\|_{\bL}\notag\\
&\leq C\bigg\|\frac{A^*_\al\vp}x\bigg\|_{\bL}\notag\\
&\leq C\|\vp\|_{\bL}\,,\label{distribute}
\end{align}
which readily yields
\begin{align*}
\bigg\|\frac u{x^2}\bigg\|_{\bL}&\leq
\bigg\|\frac{A_\al^{(2)}(\bD_{x,\al}^{(2)}u)}{x^2}\bigg\|_{\bL}+ |c|
\|x^{\al-2}f_1(\te)\|_{\bL}\\
&\leq C\|\bD_{x,\al}^{(2)}u\|_{\bL}+|c|\|x^{\al-2}\|_{\bLx}\|f_1\|_{L^2_\te}\,,\\
&\leq C\|u\|_{\bH^2_\al}
\end{align*}
provided $\al>1$, which is the condition for $x^{\al-2}$ to be in
$\bLx$. If $\al\in(0,1]$, one can easily fix the argument by multiplying by a
factor of~$x$, which yields the estimate
\[
\|u\|_{\cH^{1,1}}\leq C\|u\|_{\bH^2_\al}
\]
for $\al$ in this range.  A similar argument shows that
\begin{align*}
\bigg\|\frac{\pd_x u}x\bigg\|_{\bL}&\leq
\bigg\|\frac{A_\al^*(\bD_{x,\al}^{(2)}u)}x\bigg\|_{\bL}+C\bigg\|\frac
u{x^2}\bigg\|_{\bL}\leq C\|u\|_{\bH^2_\al}\,,\\
\|\pd_x^2u\|_{\bL}&\leq \|\bD_{x,\al}^{(2)}u\|+C\bigg\|\frac{\pd_x
  u}x\bigg\|_{\bL}+C\bigg\|\frac u{x^2}\bigg\|_{\bL}\leq C\|u\|_{\bH^2_\al}
\end{align*}
provided $\al>1$. This proves the claim for $m=2$. 

The general case follows by induction using the same argument using that if $u\in\bH^m_\al$, one can write it as
\[
u=A_\al^{(m)}(\bD_{x,\al}^{(m)}u)+\sum_{0<l\leq m/2}x^{\al+2(j-1)}f_j(\te)\,,
\]
with $\|f_j\|_{L^2_\te}\leq C\|u\|_{\bH^m_\al}$. As before, the
constraint on~$\al$ appears from the fact that, for $u$ to be in
$\cH^{k,j}$, $x^{\al-k}$ must be in $\bLx$, which forces
$\al>k-1$. The only aspect that is slightly different than above is
that the way in which the powers of~$x$ must the distributed when we
have an expression of the form $x^{-l}A_\al^{(l)}$ is by recursively
using the formulas
\[
\|x^{-l}A^*_\al\vp\|_{\bL}\leq
C\|x^{1-l}\vp\|_\bL\,,\qquad
x^{-l}A_\al\vp=\frac1xA_{\al+l-1}(x^{1-l}\vp)\,.
\]
\end{proof}

Combining Theorem~\ref{T.bH} with Proposition~\ref{P.power} we arrive
at the following useful

\begin{corollary}\label{C.bH}
If  $\al>m-l-1$,
\[
\|u\|_{\cH^{m,r}}\leq C\|x^lu\|_{\bH^{m,r}_\al}\,.
\]
\end{corollary}
\begin{proof}
It is enough to consider $l\leq m$. We then have
\begin{align*}
\|x^lu\|_{\cH^{m,r}}\leq C\|u\|_{\cH^{m-l,r+l}}\leq C\|u\|_{\bH^{m-l,r+l}_\al}\leq C\|u\|_{\bH^{m,r}_\al}\,,
\end{align*}
where we have used Theorem~\ref{T.bH} to pass to the second inequality.
\end{proof}

\section{Nonlinear estimates for adapted Sobolev spaces}
\label{S.nonlinear}

We shall next provide estimates that help us deal with nonlinear
functions of elements of an adapted Sobolev space. To obtain estimates for products of functions in adapted
Sobolev spaces, a basic result will be the following. To state it, we
will use the notation
\begin{equation}\label{cDkbe}
\cD_{k,\be}:= (x\pd_x)^k\pd_\te^\be\,.
\end{equation}

\begin{theorem}\label{T.multilinear}
Given $r>\frac{n-1}2$, consider functions $w_1,\dots, w_{m-1}\in
\cH^{1,r}$ and $u\in \cH^{0,r}$, which we can assume to be supported
in~$\cA$. Then, given multiindices with
\[
\sum_{i=1}^{m} (k_i+|\be_i|)\leq r\,,
\]
we have that
\[
\|(\cD_{k_1,\be_1}w_1)\cdots (\cD_{k_{m-1},\be_{m-1}}w_{m-1})\,
    (\cD_{k_{m},\be_{m}}u)\|_\bL\leq C\|u\|_{\cH^{0,r}}\prod_{i=1}^{m-1}
    \|w_i\|_{\cH^{1,r}}\,.
\]
\end{theorem}
\begin{proof}
Notice that for any $\al>1$ we have
\begin{align}
\Bigg\|\bigg(\prod_{j=1}^{m-1}\cD_{k_j,\be_j}w_j\bigg)&
\cD_{k_{m},\be_{m}}u\bigg\|_{\bL}^2=\int
\bigg(\prod_{j=1}^{m-1}(\cD_{k_j,\be_j}w_j)^2\bigg)\,(\cD_{k_{m},\be_{m}}
u)^2\, x\, dx\, d\te\notag\\
&\leq \int
\bigg(\prod_{j=1}^{m-1}\sup_{x'}|\cD_{k_j,\be_j}w_j(x',\te)|^2\bigg)\,(\cD_{k_{m},\be_{m}}
u)^2\, x\, dx\, d\te\notag\\
&\leq \int
\bigg(\prod_{j=1}^{m-1}\|\bD_{x,\al}\cD_{k_j,\be_j}w_j(\cdot,\te)\|_{\bLx}
\bigg)^2\,(\cD_{k_{m},\be_{m}} u)^2\, x\, dx\, d\te\notag\\
&\leq \int_{\bdry} \prod_{j=1}^{m} V_j^2\, d\te\,,\label{intVj2}
\end{align}
where we have defined
\[
V_{m}:=\|\cD_{k_{m},\be_{m}} u\|_{\bLx}\qquad \text{and}\qquad V_j:=\|\bD_{x,\al}\cD_{k_j,\be_j}w_j(\cdot,\te)\|_{\bLx}
\]
for $1\leq j\leq m-1$ and in order to pass to the third line we have
used that, by Theorem~\ref{T.AA*}, for any one-variable function
$\vp(x)\in \bH^1_\al$ with $\al>1$ we have the inequality:
\begin{align*}
\|\vp\|_{L^\infty_x}=\|A_\al(\bD_{x,\al}\vp)\|_{L^\infty_x}\leq C\|\bD_{x,\al}\vp\|_{\bLx}\,.
\end{align*}

By definition and the Sobolev embedding, when
$r-k_j-|\be_j|<\frac{n-1}2$ we have
\[
V_j\in H^{r-k_j-|\be_j|}_\te\subset L^{p_j}_\te\,,\qquad p_j:=\frac{2n-2}{n-1-2r-2k_j-2|\be_j|}\,,
\]
while for $r-k_j-|\be_j|>\frac{n-1}2$ the function $V_j$ is in
$L^\infty_\te$. For convenience, we will also relabel the functions
$V_j$ so that $r-k_j-|\be_j|>\frac{n-1}2$ if and only if $j>m'$, so
that $V_j\in L^{p_j}_\te$ with $p_j=\infty$ for $j>m'$. We will also
relabel the functions so that $r-k_j-|\be_j|=\frac{n-1}2$ exactly for
$m''< j\leq m'$, and for this range of $j$'s we will take $p_j$ to be any finite but
very large number. Of course, these last two sets can obviously be
empty. Since $\bdry$~is compact, the generalized Schwartz
inequality ensures that the integral~\eqref{intVj2} can be estimated
as
\begin{align*}
  \int_{\bdry} \prod_{j=1}^{m} V_j^2\, d\te \leq
  \prod_{j=1}^m\|V_j\|_{L^{p_j}_\te}^2\leq C\prod_{j=1}^m\|V_j\|_{H^{r-k_j-|\be_j|}_\te}^2\leq C\|u\|_{\cH^{0,r}}^2\prod_{j=1}^{m-1}\|w_j\|_{\cH^{1,r}}^2
\end{align*}
provided that
\begin{equation}\label{condVj2}
\sum_{j=1}^m\frac 2{p_j}\leq 1\,.
\end{equation}

Let us show that the condition~\eqref{condVj2} holds, which completes
the proof of the theorem. For this, let us write
\[
r=(1+\rho)\frac{n-1}2\,,
\]
where $\rho>0$ by hypothesis. Since
$p_j=\infty$ for $m>m'$ and $p_j$ is arbitrarily large for $m''<\leq
j\leq m'$, we can then take an arbitrarily small constant $\de$ such that
\begin{align}
\sum_{j=1}^m\frac 2{p_j}& \leq \sum_{j=1}^{m''}\frac 2{p_j}+\de\notag\\
&=\frac1{n-1}\sum_{j=1}^{m''}(n-1-2r+2k_j+2|\be_j|)+\de\notag\\
&=\frac1{n-1}\bigg(m''(n-1-2r)+2\sum_{j=1}^{m''}(k_j+|\be_j|)\bigg)+\de\notag\\
&\leq m''-\frac {2r(m''-1)}{n-1}+\de\notag\\
&=1-(m''-1)\rho+\de\,.\label{sumpj}
\end{align}
Therefore, the claim follows for $m''\geq 2$ by taking~$\de$ smaller
than $(m''-1)\rho$. To conclude the proof, let us discuss the
remaining cases. When $m''=0$, the claim is immediate. For
$m''=1$ one can go over the proof of~\eqref{sumpj} and observe that
the only problematic case is when $k_1+|\be_1|=r$. But in this case
$k_j+|\be_j|=0$ for all $j>1$, which implies that there are not any
$j$'s for which $r-k_j-|\be_j|=\frac{n-1}2$ and thus one can take
$\de=0$. The theorem then follows.
\end{proof}

Theorem~ \ref{T.multilinear} will be key in the rest of the paper. It
should be noticed that this theorem provides a wide range of estimates
for nonlinear functions of elements of an adapted Sobolev space. In
particular, we have the following result, where, although we do not
emphasize it notationally, here the function $F(w_1,\dots, w_N)$ can also depend on the space variables:

\begin{corollary}\label{T.Fu}
Let $u\in\cH^{0,r}$ and $w_1,\dots, w_m\in  \cH^{1,r}$ with $r>\frac{n-1}2$. Then, if $F$ is a $C^{r}$ function of~$w_j$ and a
$C^0_r$ function of the space variables (whose dependance will not be
made explicit), we have
\begin{equation}\label{Fu}
\|F(w_1,\dots, w_m)u\|_{\cH^{0,r}}\leq C \|u\|_{\cH^{0,r}}\,,
\end{equation}
where $C$ depends on $\|w_1\|_{\cH^{1,r}}+\cdots + \|w_m\|_{\cH^{1,r}}$.
\end{corollary}
\begin{proof}
The result follows by applying Theorem~\ref{T.multilinear} to the
various terms that appear after using the
Leibniz rule on
\[
\cD_{k,\be}\big[F(w_1,\dots, w_m)u\big]
\]
with $k+|\be|\leq r$.
\end{proof}



\section{Estimates for the linearized equation}
\label{S.linear}

For future convenience, we will assume that the metric~$g$ possesses
the following properties, which will be needed in the following section to prove the
convergence of the iteration set in Section~\ref{S.iteration}. While
some parameters could have been chosen in a different range for the
purposes of this section, this way the application of these results in
the following section will be transparent. 

\begin{assumption}
Throughout this section we will assume that the metric~$g$ satisfies
the following hypotheses:
\begin{enumerate}
\item The metric $g$ is weakly asymptotically AdS and can be
  written as
\[
\Bg=\Bga+xw\,,
\]
with $\ga\equiv\ga_l$ is the metric constructed in
Theorem~\ref{T.peeling} with~$l\geq\frac n2+s+2$, for some integer~$s$ satisfying
\[
2\leq s<\frac n2+2\,.
\]
We also assume that $\|\Bg^{\mu\nu}\|_{L^\infty}<\La$.

\item The tensor field~$w$ is bounded as
\begin{equation}\label{hypoii}
\sum_{k=0}^{s-1}\|\pd_t^kw\|_{L^\infty_t\cH^{2,r+s-k-2}}+ \|\pd_t^sw\|_{L^\infty_t\cH^{1,r-1}}<\La\,,
\end{equation}
for some integer $r>\frac{n-1}2$ and some constant~$\La$.

\item The metric $\Bga$ satisfies
\[
\|\Bga\|_{C^{n-1}_{p-n+1}}<\La
\]
with $p\geq l+r+s+1$, which is equivalent to demanding that the
initial and boundary data $(\tg,K,\hg)$ satisfy
\[
\|x^2\tg\|_{C^{n-1}_{p-n+1}}+ \|x^2K\|_{C^{n-1}_{p-n}}+\|\hg\|_{C^p(I\times\bdry)}<\La'\,.
\]
\end{enumerate}
\end{assumption}

Using the formula~\eqref{DQ}, which ensures that the principal part of
$P_g$ is $\bar g^{\mu\nu}\pd_\mu\pd_\nu$, together with the small-$x$ behavior
described in Lemma~\ref{L.DQVj} and the fact that~$g$ is weakly
asymptotically AdS, is easy to derive a manageable
expression for $P_gu$. Specifically, if we take $u\in
\cV^g_j$, a direct calculation shows that $P_gu$ can be
written in $\cA$ using local coordinates as
\begin{multline}\label{Pu}
(P_gu)_{\mu\nu}=-\frac12\bar g^{tt}\, \Big(\pd_t^2+
\pd_{\te^i}^*G^{ij}\pd_{\te^j}+
\bD^*_{x,\al_j}b^1\bD_{x,\al_j} + x\, \pd_{\te^i}^*
(b^2)^i\pd_x\\
+x\pd_xb^3\pd_t+\pd_{\te^i}(b^4)^i\pd_t\Big)\, u_{\mu\nu}\\
+\Big(b^5  \pd_xu+ b^6 \pd_tu +  b^7
\pd_{\te}u + \frac{b^8}x u\Big)_{\mu\nu}\,,
\end{multline}
where as usual the local coordinates $\te=(\te^1,\dots,\te^{n-1})$ parametrize the
boundary $\bdry$, the star denotes the formal adjoint of a
differential operator computed with respect to the
scalar product of $\bL$, and the quantities $b^l$ are scalar functions or tensor
fields that depend smoothly on $\hg$, $\pd\hg$ (through~$\ga$
and~$\pd\ga$), $w$ and $\pd w$. Observe that the
principal part of~$P_g$ is scalar. Although we do not
make explicit the tensorial structure of the tensor fields $b^l$
appearing in the non-principal part of the operator, their action must
be understood in the obvious fashion, e.g., 
\[
(b^6  \pd_xu)_{\mu\nu}\equiv (b^6)^{\la\rho}_{\mu\nu} \,  \pd_xu_{\la\rho}
\]
Notice that, in particular, 
\begin{equation}\label{bs}
b^1=-\frac{\Bg^{xx}}{\Bg^{tt}}\,,\quad
G^{ij}=-\frac{\Bg^{\te^i\te^j}}{\Bg^{tt}}\,,\quad
 x(b^2)^i=-\frac{2\Bg^{x\te^i}}{\Bg^{tt}}\,, \quad
   xb^3=-\frac{2\Bg^{xt}}{\Bg^{tt}}\,, \quad (b^4)^i=-\frac{2\Bg^{t\te^i}}{\Bg^{tt}}\,.
 \end{equation}
 Since the metric is weakly asymptotically AdS, all the quantities
$b^j$ are of order $\cO(1)$, with $b^1>0$ and $G^{ij}$ a positive
definite matrix.

We shall next derive estimates for a function satisfying
the scalar equation
\begin{equation}\label{LuF}
L_{g,\al}u=F\,, \qquad u|_{t=0}=u_0\,,\qquad \pd_t u|_{t=0}=u_1\,,
\end{equation}
where 
\[
L_{g,\al}u:=\Big(\pd_t^2+
\pd_{\te^i}^*G^{ij}\pd_{\te^j}+
\bD^*_{x,\al}b^1\bD_{x,\al} + x\, \pd_{\te^i}^*
(b^2)^i\pd_x\\
+x\pd_xb^3\pd_t+\pd_{\te^i}(b^4)^i\pd_t\Big)u\,.
\]
Taking $\al=\al_j$, $L_{g,\al}$ would be the part of $P_gu$ containing both the highest order derivatives and
the more singular terms at~$x=0$, which is a scalar
differential operator for $u\in \cV^g_j$. The metric~$g$
is assumed to satisfy the
above hypotheses, and we will also assume that $\al\geq n/2$. The reason for which we introduce this auxiliary
equation is to postpone the treatment of the tensorial nature of the
equation until the end of this section, but we have chosen to keep the
notation~$u$ for the unknown as we will eventually replace~$u$ by a
tensor field satisfying $P_gu=F$.

In the following theorem we provide a priori estimates for the
problem~\eqref{LuF}. 
To state the theorem in a notationally concise way, let us denote by 
\begin{equation}\label{uk}
u_k:=\pd_t^ku|_{t=0}\,, \qquad 2\leq k\leq s,
\end{equation}
the value of the $k^{\text{th}}$ time derivative of~$u$
at~$t=0$. Notice that, by isolating the term with the highest number of
time derivatives in~\eqref{LuF} and differentiating $k-2$ times with respect
to~$t$, one can write $u_k$ in terms of
derivatives of the initial data and source term $(u_0,u_1,F)$. The
functions~$u_k$ will often appear in arguments via the quantity
\begin{equation}\label{cC}
\cC_{s,r}:= \sum_{k=0}^{s-1}\|u_k\|_{\cH^{1,r+s-k-1}}+\|u_s\|_{\cH^{0,r}}\,.
\end{equation}
For the tensor-valued equation $P_gu=F$, this quantity will correspond
to the quantity that appears in the statement of Theorem~\ref{T.main}.

To state the results, we will make use of the following norms (here the prime does not refer to any sort of duality):
\begin{subequations}\label{norms}
\begin{align}
\|u\|_{s,r}&:=\sum_{k=0}^{s-1}\|\pd_t^k u\|_{L^\infty_t\cH^{1,r+s-k-1}}+\|\pd_t^su\|_{L^\infty_t\cH^{0,r}}\,,\\
\|F\|_{s,r}'&:= \sum_{k=0}^{s-1}\|\pd_t^kF\|_{L^\infty_t\cH^{0,r+s-k-1}}\,.
\end{align}
\end{subequations}
Throughout, we will use the notation $C_0$ for constants 
depending only on~$\de$ and~$\La$.

\begin{theorem}\label{T.linealL}
For any $F\in L^\infty_t\bL$ there is a unique solution $u\in
L^\infty_t\cH^1\cap W^{1,\infty}_t\bL$ to the Cauchy problem~\eqref{LuF}, which
satisfies the following estimate in $(-T,T)\times\bola$:
\[
\|u\|_{s,r}\leq
\e^{C_0 T}\cC_{s,r}+ C_0 T\|F\|_{s,r}'\,.
\]
For small $T$, the constant depends only on $\La$.
\end{theorem}
\begin{proof}
It is standard that it suffices to prove the a priori
estimate. For this, there is no loss of generality in
assuming that~$u$ is supported in~$\cA$, since the estimate is 
known to hold for~$u$ supported away from the boundary. Let us then
define the energy functional 
\begin{equation}\label{E1}
E_1[v]:=\frac12\int_{\bola}\big((\pd_t v)^2+G^{ij}\pd_iv\,\pd_j
v+b^1(\bD_{x,\al}v)^2+x(b^2)^i\pd_xv\, \pd_iv\big)\, x\, dx\, d\te\,,
\end{equation}
where in the rest of this section we will write
$\pd_i\equiv\pd_{\te^i}$. It is apparent that at any time $E_1[v]^{\frac12}$ is equivalent to the norm
$\|v\|_{\bH^1_\al}+\|\pd_t v\|_\bL$ (which is in turn equivalent to
$\|v\|_{\cH^1}+\|\pd_t v\|_\bL$ by Theorem~\ref{T.bH}) in the sense that
\begin{equation}\label{normE1}
\frac{1}{C } E_1[v]^{\frac12}\leq \|v\|_{\cH^1}+\|\pd_tv\|_\bL\leq C E_1[v]^{\frac12}\,,
\end{equation}
where the constant~$C$ only depends on
\begin{equation*}
\|\Bg\|_{C^1_1}+\|\pd_t\Bg\|_{C^0_1}+\|\pd_t^2\Bg\|_{C^0}\,.
\end{equation*}
In particular, by Corollary~\ref{C.Cmr}, $C\equiv C_0$ only depends on~$\La$.

Now let us use the energy functional~\eqref{E1} to define
\[
E_{1,r'}[v]:=\sum_{k+|\be|\leq r'}E_1[\cD_{k,\be} v]\,,
\]
where again we are using the shorthand notation $\cD_{k,\be}:= (x\pd_x)^k\pd_\te^\be$.
In view of the norm equivalence~\eqref{normE1}, it is clear that
$E_{1,r'}[v]$ is equivalent to the norm
\[
\|v\|_{\cH^{1,r'}}+\|\pd_t v\|_{\cH^{0,r'}}
\]
with a constant that only depends on~$\La$. We can now define a higher analog of the energy~$E_1$
by setting
\begin{equation}\label{E3r}
E_{s,r}[v]:=\sum_{k=0}^{s-1}E_{1,r+s-k-1}[\pd_t^kv]\,.
\end{equation}
In view of the norm equivalence~\eqref{normE1}, it is clear that $E_{s,r}[v]^{1/2}$ is equivalent to the norm
\begin{equation}\label{normE3r}
\sum_{k=0}^{s-1}\|\pd_t^kv\|_{\cH^{1,r+s-k-1}}+ \|\pd_t^s v\|_{\cH^{0,r}}
\end{equation}
in the same sense as above, which implies that
\[
\sup_{|t|<T}E_{s,r}[v]^{\frac12}
\]
is equivalent to $\|v\|_{s,r}$.

Our goal now is to show that, if $u$ is a solution of~\eqref{LuF}, the
energy $E_{s,r}[u]$ satisfies the differential inequality
\begin{equation}\label{est2}
\pd_t E_{s,r}[u]\leq C_0  E_{s,r}[u]+C_0  E_{s,r}[u]^{\frac12}\sum_{k=0}^{s-1}\|\pd_t^kF\|_{\cH^{0,r+s-k-1}}\,.
\end{equation}
Indeed, by Gr\"onwall's inequality it is standard that this implies 
\[
E_{s,r}[u](t)^{\frac12}\leq \e^{C_0' |t|}\bigg(E_{s,r}[u](0)^{\frac12}+C_0'\sum_{k=0}^{s-1}\int_{-|t|}^{|t|}\|\pd_t^kF\|_{\cH^{0,r+s-k-1}}\bigg)\,.
\]
Since $E_{s,r}[u]^{\frac12}$ is equivalent to $\|u\|_{s,r}$, the a priori estimate of the theorem then follows from the above inequality.

Armed with Theorems~\ref{T.Morrey} and~\ref{T.multilinear}, the proof
of~\eqref{est2} is now standard. Let us begin by computing the
evolution of $ E_{1,r+s-1}[u]$. One readily finds that it is given by
\begin{multline}\label{largo}
\pd_t E_{1,r+s-1}[u]= \sum_{k+|\be|\leq r+s-1} \bigg[\int \pd_t(\cD_{k,\be} u)\, L_{g,\al}(\cD_{k,\be}
u)\\-\int x\pd_t\cD_{k,\be} u\,\pd_x(b^3\pd_t\cD_{k,\be} u)
-\int
\pd_t\cD_{k,\be} u\,\pd_{i}((b^4)^i\pd_t\cD_{k,\be} u)
\\+\int \cO(1)\pd_t\cD_{k,\be} u\, \pd
\cD_{k,\be} u+\int\frac{\cO(1)}x\cD_{k,\be} u\, \pd_t\cD_{k,\be}
u\\
+\int\frac{\cO(1)}x\, (\cD_{k,\be}u)^2+ \int\cO(1)\, (\pd\cD_{k,\be}u)^2\bigg]\,,
\end{multline}
where all the integrals hereafter correspond to integration over the ball with respect to the
natural measure $x\, dx\, d\te$ and we are denoting by $\cO(1)$ well-behaved
functions of~$\Bga$, $w$ and $\pd w$. 
We claim that this can be estimated as
\begin{equation}\label{E1b}
\pd_t E_{1,r+s-1}[u]\leq C_0 E_{1,r+s-1}[u]+C_0 E_{1,r+s-1}[u]^{\frac12}\sum_{k+|\be|\leq r+s-1} \|L_{g,\al}(\cD_{k,\be}
u)\|\,,
\end{equation}
where~$\|\cdot\|$ stands for the $\bL$~norm. Indeed, for $k+|\be|\leq
r+s-1$ the first term
in~\eqref{largo} is bounded as
\[
\int |\pd_t\cD_{k,\be} u\, L_{g,\al} \cD_{k,\be} u| \leq C_0
E_{1,r+s-1}[u]^{\frac12}\, \|L_{g,\al} \cD_{k,\be} u\|
\]
and the last for summands can be easily upper bounded by
\[
C_0 E_{1,r+s-1}[u]
\]
using Theorems~\ref{T.Morrey} and~\ref{T.multilinear}. 
Let us now consider the first of the two remaining terms. We have that
\begin{align*}
\bigg|\int x\pd_t \cD_{k,\be} u\,\pd_x(b^3\pd_t \cD_{k,\be} u) \bigg|&=\bigg|\int (\pd_t
\cD_{k,\be} u)^2x\pd_xb^3+\frac12\int b^3x\pd_x[(\pd_t \cD_{k,\be} u)^2]\bigg|\\
&\leq\int\Big|\frac12x\pd_x b^3-b^3\Big|(\pd_t \cD_{k,\be} u)^2\\
&\leq C_0 E_{1,r+s-1}[u]
\end{align*}
and an analogous argument shows that
\[
\bigg|\int
\pd_t\cD_{k,\be} u\,\pd_{i}((b^4)^i\pd_t\cD_{k,\be} u) \bigg|\leq C_0 E_{1,r+s-1}[u]\,.
\]
Putting everything together, this yields~\eqref{E1b}. To conclude, we
can now estimate the commutator using Theorems~\ref{T.Morrey} and~\ref{T.multilinear} to infer that
\begin{align*}
  \|L_{g,\al}(\cD_{k,\be} u)\|
&\leq \|\cD_{k,\be}
  (L_{g,\al}u)\|+\|[L_{g,\al},\cD_{k,\be}]u\|\\
&\leq \|\cD_{k,\be}
F\|+\|[L_{g,\al},\cD_{k,\be}]u\|\\
&\leq \|F\|_{\cH^{0,r+s-1}}+C_0 E_{s,r}[u]^{\frac12}\,,
\end{align*}
which shows that
\[
\pd_t E_{1,r+s-1}[u]\leq C_0 E_{s,r}[u] + C_0 E_{s,r}[u]^{\frac12}\|F\|_{s,r}'\,.
\]
The computation of the time evolution of the other quantities
$E_{1,r+s-k-1}[\pd_t^k u]$
appearing in the definition of $E_{s,r}[u]$ (cf.\ Equation~\eqref{E3r}) is
similar, the only difference being that one needs to control the
commutator
\begin{align*}
  \|L_{g,\al}(\cD_{j,\be}\pd_t^k u)\|&\leq \|\cD_{j,\be}
\pd_t^k F\|+\|[L_{g,\al},\cD_{j,\be}\pd_t^k]u\|\\
&\leq \|\pd_t^kF\|_{\cH^{0,r+s-k-1}}+C_0 E_{s,r}[u]^{\frac12}\,.
\end{align*}
Summing over $k$, this readily yields the differential inequality~\eqref{est2}.
\end{proof}

\begin{remark}
  Notice that we are not imposing that $u(t)\in \bH^2_\al$ for
  a.e.~$t$, so Equation~\eqref{LuF} has to be understood using the
  energy formulation, as it is customary.
\end{remark}

Promoting the estimates proved in Theorem~\ref{T.linealL} to estimates
for the tensor-valued equation
\begin{equation}\label{PuF}
P_gu=F\,, \qquad u|_{t=0}=u_0\,,\qquad \pd_t u|_{t=0}=u_1\,,
\end{equation}
is now immediate as the norms~\eqref{norms} can be trivially extended
to tensor-valued functions. As before, we will state the theorem in
terms of the quantity~$\cC_{s,r}$, which we can still define in terms
of the initial data and source term at $t=0$ as in Equation~\eqref{cC}.

\begin{theorem}\label{T.linear}
For all times $T<T_0$, if~$u$ solves the problem~\eqref{PuF} one has
the estimates
\begin{equation*}
\|u\|_{s,r}\leq e^{C_0T}\cC_{s,r}+ C_0 T\|F\|_{s,r}' \,,
\end{equation*}
where the constant $C_0$ depends only on~$\La$.
\end{theorem}

A final simple result that will come in handy in the following section
is the following, which controls the difference between the solution
to two Cauchy problems of the form~\eqref{PuF} with different metrics
and source terms. For concreteness we will control the difference in
the $\|\cdot\|_{1,0}$ norm and assume that we have the same initial
conditions $(u_0,u_1)$, but we could have used any norm
$\|\cdot\|_{s',r}$ with $s'\leq s-1$ and allowed for distinct initial
conditions. It is worth emphasizing that estimating the difference is
not completely trivial a priori because the leading part of the equation, as
represented by the operator $P_g$, is not scalar: we have seen that
the parameter $\al=\al_j$ takes a different value depending on the
subspace $\cV_j^g$ that $u$ is assumed to belong to. However, the
structure of the metrics under consideration allows to prove the
result quite easily.

\begin{proposition}\label{P.difference}
Let 
\[
\Bg:=\Bga + xw\qquad \text{and} \qquad \Bg':=\Bga + xw'
\]
be metrics satisfying the assumptions~(i)--(iii)
above. Suppose that $u,u'\in L^\infty\cH^1\cap H^1_t\bL$ satisfy the equations
\[
P_gu=F\qquad\text{and} \qquad P_{g'}u'=F'
\]
with the same initial conditions $(u_0,u_1)$. Then the difference is bounded by
\begin{equation*}
  \|u-u'\|_{1,0}\leq C \e^{CT} T(\|F-F'\|_{1,0}'+\|w-w'\|_{1,0})\,,
\end{equation*}
where the constant $C$ only depends on~$\La$, $\|F\|_{s,r}'$ and $\cC_{s,r}$.
\end{proposition}

\begin{proof}
  A short computation using the expression for~$P_g$ shows that the differential operator $P_g$, whose
  leading part at $x=0$ is {\em not}\/ scalar, can be symbolically
  written in a neighborhood of $x=0$ as
\begin{equation}\label{Pgga}
P_g u=A_2(\Bg)\, \pd^2 u+\bigg(\frac{A_1(\Bg)}x+A_1'(\Bg,\pd\Bg)\bigg)\,\pd u+
\bigg(\frac{A_0(\Bg)}{x^2}+\frac{A_0'(\Bg,\pd\Bg)}x\bigg)\, u\,,
\end{equation}
where $A_j,A_j'$ are tensor-valued functions. Furthermore, we know that the term with second-order derivatives is
scalar, and given by~\eqref{Pu}. 

With $\Bg=\Bga+xw$, it then follows that $P_g$ agrees with $P_\ga$
modulo terms that are subdominant at $x=0$. More precisely, Theorem~\ref{T.multilinear} yields
\begin{align}
\|(P_{g}-P_{g'})u\| &\leq \sum_{k=0}^2
\bigg\|\frac{A_k(\Bg)-A_k(\Bg')}{x^{2-k}}\pd^k u \bigg\|+ \sum_{0}^1 \bigg\|\frac{A'_k(\Bg,\pd\Bg)-A'_k(\Bg',\pd\Bg')}{x^{1-k}}\pd^k u \bigg\|\notag\\
&\leq C\|w-w'\|_{1,0}\,,\label{P-Pbound}
\end{align}
with $\|\cdot\|$
denoting the $\bL$~norm and the constant~$C$ depending only on the
quantities discussed at the statement as a consequence of the 
estimates for~$u$ proved in Theorem~\ref{T.linear}. 

To see why this is true, let us consider a term that does not depend
on $\pd\Bg$, such as $A_2(\Bg)\,\pd^2 u$. Observe that, as the
$L^\infty$ norm of~$w$ and~$\pd w$ is bounded by a constant that
depends on~$\La$ by Theorem~\ref{T.Morrey}, it is standard that we
have
\[
|A(\Bg,\pd\Bg)-A(\Bg',\pd \Bg')|\leq C_0(|w-w'| + x|\pd w-\pd w'|)\,.
\]
Therefore,
\begin{align*}
\|(A_2(\Bg)-A_2(\Bg'))\, \pd^2 u\|&\leq \|(xw-xw')H(xw,xw')\pd^2 u\|\\
& \leq C\|(w-w') x\pd^2 u\|\\
& \leq C\|w-w'\|_{L^\infty_x L^2_\te}\|x\pd^2 u\|_{\bLx L^\infty_\te}\\
& \leq C\|w-w'\|_{\cH^1}\|u\|_{\cH^{1,r'+1}}\\
& \leq C\|w-w'\|_{1,0}\,.
\end{align*}
Here $H$ is a smooth tensor-valued function, $r'$ is any number larger
in $(\frac{n-1}2,r]$ and the constant~$C$ is as above. When 
derivatives of $\Bg$ are involved, the argument is similar. For instance,
\begin{align*}
\|(A_1'(\Bg,\pd\Bg)-A_1'(\Bg',\pd\Bg'))\, \pd u\|&\leq \|(w-w')H_1\pd
u\|+\|x(\pd w-\pd w')H_2\pd u\|\\
& \leq C\|(w-w') \pd u\|+ C\|(\pd w-\pd w') x\pd u\|\\
& \leq C\|w-w'\|_{L^\infty_x L^2_\te}\|\pd u\|_{\bLx L^\infty_\te} +
C\|\pd w-\pd w'\| \|x\pd u\|_{L^\infty}\\
& \leq C\|w-w'\|_{\cH^1}\|u\|_{\cH^{1,r'+1}}\\
& \leq C\|w-w'\|_{1,0}\,.
\end{align*}

To conclude the proof of the proposition, let us notice that
\[
P_{g'}(u-u')=F'-F+(P_{g}-P_{g'})u\,.
\]
Since
\[
\|(P_{g}-P_{g'})u\|_{1,0}'=\|(P_{g}-P_{g'})u\|_{L^\infty_t\bL}\leq C\|w-w'\|_{1,0}
\]
by~\eqref{P-Pbound}, Theorem~\ref{T.linear} then provides the desired control for the difference~$u-u'$.
\end{proof}

\section{Convergence of the iteration}
\label{S.convergence}

We are now ready to prove the existence of solutions to the equation
$Q(g)=0$ with the desired initial and boundary conditions. With the
technical tools that we have already developed, the argument is now
standard. 

To present the result, let us introduce a new norm that is stronger
than $\|u\|_{s,r}$ in the sense that it also includes additional (adapted)
derivatives with respect to the variable~$x$. To define it, we can
assume that the tensor field $u$ is supported in~$\cA$ and consider
its decomposition
\[
u=u^0+u^1+u^2+u^3\,,
\]
where $u^j\in\cV_j^\ga$. The norm is then defined using the
metric~$\ga$ as
\[
\triple u_{s,r}:=\|u\|_{s,r}+\sum_{j=0}^3\;\sum_{i+k+m\leq s-2}\|\bD_{x,\al_j}^{(2+i)}\pd_t^ku^j\|_{\cH^{0,r+m}}\,.
\]
For $s=1$ we simply take $\triple u_{1,r}:=\|u\|_{1,r}$. By
Theorem~\ref{T.bH} and the fact that $\al_j\geq \frac n2$, for $s<\frac n2+1$ this is equivalent to
\[
\triple u_{s,r}:=\|u\|_{s,r}+\sum_{i+k+m\leq s-2}\|\pd_t^ku\|_{\cH^{2+i,r+m}}\,,
\]
so in particular it does not depend on~$\ga$.  Likewise, for $s\in [\frac
n2+1,\frac n2+2)$ one can write
\[
\triple u_{s,r}:=\|u\|_{s,r}+\sum_{j=0}^3
\|\bD_{x,\al_j}^{(s)}u^j\|_{\cH^{0,r}}+ \sum_{i+k+m\leq
    s-2  \text{ and } i\leq s-3}\|\pd_t^ku\|_{\cH^{2+i,r+m}}\,,
\]
Of course, when $u$ is not supported in~$\cA$ one defines its
triple norm using a compactly supported function~$\chi$ e.g.\ as in
Equation~\eqref{bH}. It should be noticed that we will not only estimate
$u$, but also $x^\rho u$, as in the bound~\eqref{x-triple} below. The reason for
this is that this not only amounts to redistributing standard and
regularized derivatives as in Proposition~\ref{P.power}, but in fact
allows us to control~$\rho$ additional time derivatives of~$u$. This
will be useful to prove Theorem~\ref{T.main}.

\begin{theorem}\label{T.iteration}
Let us choose numbers $s$,  $r$, $l$ and~$p$  and take~$\ga\equiv \ga_l$ as in the
assumptions~(i)--(iii) of Section~\ref{S.linear}. For any compatible
initial and boundary data $(\tg,K,\hg)$, there is some time
$T>0$ and a function~$u$ such that the weakly asymptotically AdS metric
\[
g:=\ga+x^{\frac n2}u
\]
solves the modified Einstein equation $Q(g)=0$ in $(-T,T)\times\bola$
with the specified initial and boundary conditions and is bounded as
\begin{equation}\label{bound-triple}
\triple u_{s,r}< C
\end{equation}
with a constant depending only on
\[
\|x^2\tg\|_{C^{n-1}_{p-n+1}}+
\|x^2K\|_{C^{n-1}_{p-n}}+\|\hg\|_{C^p(I\times\bdry)}\,. 
\]
Furthermore, if $r>\frac{n-1}2+\rho$ with $\rho$ a positive integer,
we also have
\begin{equation}\label{x-triple}
\triple{x^\rho u}_{s+\rho,r-\rho}< C\,,
\end{equation}
and $\Bg\in C^\infty\poly$ if $\tg\in C^\infty$ and $x^2\tg, x^2K\in C^\infty\poly$.
\end{theorem}

\begin{proof}
For simplicity we will divide the proof in four steps. As usual, it is
enough to prove the estimates in a small neighborhood~$\cA$ of the
boundary. As before, we will write the metric as $\Bg=\Bga+x^{\frac
  n2+2}u$ and write the equation $Q(g)=0$ in the convenient
form~\eqref{EinsteinP}.\smallskip

\noindent{\em Estimates for the source terms.} Let us
begin by deriving some estimates for the functions~$\cF(u)$
and~$\cE(u)$ under the assumptions that 
\begin{equation}\label{ass-uX}
\|\Bga\|_{C^{n-1}_{p-n+1}}<\La,\quad \|u\|_{s,r}<\La,\quad \|\Bg^{\mu\nu}\|_{L^\infty}< \La\,;
\end{equation}
cf.\ Section~\ref{S.linear}. Just as in that section, we
will write the metric as 
$\Bg=\Bga+ xw$ with $w:=x^{\frac n2+1}u$ bounded in the norm~\eqref{hypoii}.
Throughout, we will denote by $C_0$ a
constant that only depends on~$\La$ and~$\de$ and we will use without
further mention the properties of the adapted Sobolev spaces that we
established in Sections~\ref{S.Sobolev} and~\ref{S.nonlinear}.

A close look at Equation~\eqref{tcF} reveals that the 
function $\cF(u)$ can be written as
\[
\cF(u)=\frac{F(\Bg)u}x\,,
\]
where $F(\Bg)$ is a smooth function of~$\Bg:= \Bga + x^{\frac n2+2}u$ (in particular, $\cF(u)$ does
not involve any derivatives of~$u$). Hence at any fixed time we have
\[
\|\cF(u)-\cF(u')\|\leq C_0\bigg\|\frac{u-u'}x\bigg\|\leq C_0\|u-u'\|_{\cH^1}\,,
\]
where $\|\cdot\|$ again stands for the $\bL$ norm, which implies
\begin{equation*}
\|\cF(u)-\cF(u')\|_{1,0}'\leq C_0\|u-u'\|_{1,0}\,.
\end{equation*}
Furthermore, by the elementary inequality $\|v/x\|_{\cH^{k,s}}\leq C\|v\|_{\cH^{k+1,s}}$,
\begin{align*}
\|\cF(u)\|_{s,r}'&= \sup_{|t|<T}\sum_{k=0}^{s-1}\|\pd_t^k\cF(u)\|_{\cH^{0,r+s-k-1}} \notag\\
& \leq
C_0\sup_{|t|<T}\sum_{k=1}^{s-1}\|\pd_t^ku\|_{\cH^{1,r+s-k-1}}
\notag\\
&\leq C_0\|u\|_{s,r}\,.
\end{align*}

Using the formula for $\cE(u)$ given in Equation~\eqref{tcE} and computing
the second derivative of $B$ as in Lemma~\ref{L.D2Q}, we infer that
$\cE(u)$ can be symbolically written as
\[
\cE(u)=\int_0^1x^{\frac n2} B(u,x\,\pd u)\, d\si\,,
\]
where $B$ is a quadratic form whose coefficients are smooth functions
of $\Bga+ \si x^{\frac n2+2} u$ and the integral is with respect to
the parameter~$\si$. Using this formula and arguing essentially as in the case of
$\cF(u)$ one can prove the analogous estimates
\begin{align*}
\|\cE(u)-\cE(u')\|_{1,0}'&\leq
C_0\|u-u'\|_{1,0}\,,\\
\|\cE(u)\|_{s,r}'&\leq C_0\|u\|_{s,r} \,.
\end{align*}
Hence it stems that the function $\cG(u):=\cF(u)+\cE(u)$ that appears
in Equation~\eqref{EinsteinP} satisfies the same bounds, that is,
\begin{align}
\|\cG(u)-\cG(u')\|_{1,0}'&\leq
C_0\|u-u'\|_{1,0}\label{Lip-cG}\,,\\
\|\cG(u)\|_{s,r}'&\leq C_0\|u\|_{s,r} \label{bound-cG}\,.
\end{align}

\smallskip

\noindent{\em Convergence in the low norm.} Our objective will be to
solve the equation using the iteration
\begin{subequations}\label{iteration}
\begin{equation}
P_{g^{m}}u^{m+1}=\cF_0+\cG(u^{m})\,,
\end{equation}
where $g^m:= \ga+ x^{\frac n2}u^m$ and the initial conditions
that we impose are
\begin{equation}
u^{m+1}|_{t=0}=u_0\,,\qquad \pd_t u^{m+1}|_{t=0}=u_1\,,
\end{equation}
\end{subequations}
where of course $u_j:=x^{-\frac n2}(g_j-\pd_t^j\ga|_{t=0})$. We can start the
iteration with $u^1:=0$ and the desired solution to the equation
$Q(u)=0$ will arise as the limit of $u^m$ as $m\to\infty$. Notice that
we are using superscripts both for the sequence of iterates and for the
components of $u$ in the space $\cV_j^\ga$, but this should not cause
any confusion because only the former will appear in the study of the
convergence of the sequence.

Let us assume that the condition~\eqref{ass-uX} is satisfied, where
$\La$ is chosen so that
\begin{equation}\label{La2}
\|\Bga\|_{C^{n-1}_{p-n+1}}+\|\Bg^{\mu\nu}|_{t=0}\|_{L^\infty}+\cC_{s,r}+\|\cF_0\|_{s,r}'<\frac\La2\,.
\end{equation}
Recall that, by Theorem~\ref{T.peeling},
\[
\|\cF_0\|_{s,r}'\leq
C\|\cF_0\|_{C^1_{s+r-1}(I\times\bola)}\leq C
\|\hg\|_{C^p(I\times\bdry)}\,,
\]
where we have used that $l\geq s+\frac n2+2$ and $p\geq l+s+r+1$, so this just means that we choose $\La$ in terms of the sizes of the initial and
boundary data.

To prove the convergence of the sequence in the
norm~$\|\cdot\|_{1,0}$, then we can use
Proposition~\ref{P.difference} and the estimate~\eqref{Lip-cG} to write, for $T<T_0$,
\begin{align}
\|u^{m+1}-u^m\|_{1,0}&\leq C_0T\|\cG(u^m)-\cG(u^{m-1})\|_{1,0}' +C_0T
\|u^m-u^{m-1}\|_{1,0}\notag\\
&\leq C_0T\|u^m-u^{m-1}\|_{1,0}\,. \label{lowbound}
\end{align}
It then follows that the sequence $(u^m)_{m=1}^\infty$ converges in
the norm~$\|\cdot\|_{1,0}$ to some $u\in L^\infty_t\cH^1\cap
W^{1,\infty}_t \bL$, provided that $T$ is smaller than some constant
depending only on~$\La$ (i.e., $T<1/(2C_0)$).

\smallskip

\noindent{\em Boundedness in the high norm.} Let us assume that the
bound~\eqref{ass-uX} is satisfied up to the $m^{\text{th}}$ step of the iteration with $\La$ chosen so
that~\eqref{La2} holds. Writing $g^m=\ga+x w^m$ with
$w^m:=x^{\frac n2+1}u^m$, we then
infer that the assumptions on the metric of Section~\ref{S.linear} are
satisfied too. Hence applying Theorem~\ref{T.linear} to
Equation~\eqref{iteration} immediately yields, for $T<T_0$,
\begin{align}\label{um1}
\|u^{m+1}\|_{s,r}\leq e^{C_0T}\cC_{s,r}+
C_0 T\big(\|\cF_0\|_{s,r}'+\|\cG(u^{m})\|_{s,r}'\big)\,.
\end{align}
If
we employ that $\|\cF_0\|_{s,r}'<\La/2$ in the inequality~\eqref{um1} and use the
estimate~\eqref{bound-cG}, we arrive at
\begin{align}
\|u^{m+1}\|_{s,r}&\leq e^{C_0T}\cC_{s,r}+C_0T\|u^m\|_{s,r}\notag\\
&\leq (e^{C_0T}+C_0T)\frac\La2\notag\\
&< \La \label{highbound}
\end{align}
provided that $T$ is small enough. 

Since the sequence $(u^m)$ is bounded in $\|\cdot\|_{s,r}$
by~\eqref{highbound} and converges to $u$ in~$\|\cdot\|_{1,0}$
by~\eqref{lowbound}, together with the fact that these spaces possess good
interpolation properties (essentially as a
consequence of the formula~\eqref{cHs}), we immediately obtain that
$u^m\to u$ in $\|\cdot\|_{s',r}$ for any real $s'<s$ and that~$u$ also
satisfies the bound $\|u\|_{s,r}\leq \La$. The usual argument then shows (cf.\
e.g.~\cite[Chapter~9]{Ringstrom}) that~$u$ is indeed a solution of the equation
$Q(g)=0$ in $(-T,T)\times \bola$, with $T$ small enough, and that~$u$
is bounded by
\begin{equation}\label{goodbound-u}
\|u\|_{s,r}<\La
\end{equation}
as a consequence of~\eqref{highbound}.

\smallskip

\noindent{\em Higher spatial regularity.} Our
goal now is to show that, if $u$ satisfies the equation $Q(g)=0$,
up to $s$ adapted derivatives of~$u$ can
then be controlled in terms of the energy $E_{s,r}[u]$. More
precisely, need to prove that
\begin{equation}\label{xtraders}
\sum_{j=0}^3\;\sum_{i+k+m\leq
  s-2}\|\bD_{x,\al_j}^{(2+i)}\pd_t^ku^j\|_{L^\infty_t\cH^{0,r+m}}\leq
C_0\|u\|_{s,r} + C_0\sum_{i+k+m\leq s-2}\|\pd_t^k \cF_0\|_{\cH^{i,r+m}}\,.
\end{equation}
Since $l\geq s+\frac n2+2$ and $p\geq l+s+r+1$,
Theorem~\ref{T.peeling} then asserts that
\begin{align*}
\sum_{i+k+m\leq s-2}\|\pd_t^k \cF_0\|_{\cH^{i,r+m}}&\leq
C\|x^{2-s}\cF_0\|_{C^0_{s+r-2}(I\times\bola)}\\
&\leq C \La\,.
\end{align*}
Hence the desired bound~\eqref{bound-triple} follows from the inequality~\eqref{xtraders}
and the estimate~\eqref{goodbound-u}. 

The estimates~\eqref{xtraders} are proved by isolating the term
$\bD_{x,\al}^{(2)}u$ in the equation~$Q(g)=0$, which we write as
\[
P_gu=\cF_0+\cG(u)
\]
with $g=\ga+ x^{\frac n2}u$.  Once the term $\bD_{x,\al}^{(2)}u$ has
been isolated, we can take the necessary number of adapted $x$-derivatives
for which we need a priori estimates. For
concreteness, let us spell out the details for the first quantity,
namely the norm $\|\bD_{x,\al}^{(2)}u\|_{L^\infty_t\cH^{0,r+s-2}}$.

From Equation~\eqref{LuF} we
can write
\begin{multline}\label{isolate}
\bD_{x,\al_j}^{(2)}u^j = \frac1{b^1}\Big(\cF_0^j+ \cG(u)^j+\pd_t^2 u^j -(\pd_x b^1)\bD_{x,\al_j}u^j
-\pd_i^*(G^{ik}\pd_k u^j)-x\pd_i^*[(b^2)^i\pd_x u^j]\\
-x\pd_x(b^3\pd_t u^j)-\pd_i[(b^4)^i\pd_ t u^j]+\lot\Big)\,,
\end{multline}
where the superscript~$j$ indicates the component in $\cV^\ga_j$ and we have employed the identity~\eqref{Pgga} to write
\[
P_gu=P_\ga u+\lot
\]
using the same ideas as in the proof of Proposition~\ref{P-Pbound}.
Besides, we have used that, as thanks to our choice of the number
$s,r$ we have the uniform bound
\[
\|\Bg-\Bg|_{t=0}\|_{L^\infty}\leq CT\,,
\]
Equation~\eqref{bs} guarantees that we can indeed divide by $b^1$ to solve
the equation for $\bD_{x,\al_j}^{(2)}u$. To compute the norm
$\|\bD_{x,\al_j}^{(2)}u^j\|_{\cH^{0,r+s-2}}$ we must now consider the
action of the differential operator $\cD_{k,\be}$ on this equation,
with $l+|\be|\leq r+s-2$ and $\cD_{k,\be}$ defined as
in~\eqref{cDkbe}. Given the dependence on~$u$ of the various terms
that appear in the equation, a straightforward computation shows that
in fact the terms that appear can indeed be controlled
using the norm~$\|u\|_{s,r}$ and Theorems~\ref{T.Morrey}
and~\ref{T.multilinear} as
\begin{equation}
\|\bD_{x,\al_j}^{(2)}u^j\|_{L^\infty_t\cH^{0,r+s-2}} \leq
C_0\|u\|_{s,r}+C_0 \|\cF_0\|_{L^\infty_t\cH^{0,r+s-2}} \,.
\end{equation}
Although we will not write down the tedious but straightforward
minutiae, it is clear from~\eqref{isolate}, e.g., that the most
dangerous terms that can appear when one estimates
$\|\bD_{x,\al_j}^{(2)}u^j\|_{\cH^{0,r+s-2}}$ are of the symbolic form
\begin{align*}
  \|F(u)\pd_t^2 u\|_{\cH^{0,r+s-2}}+ \|F(u)x\pd_x\pd_\te
  u\|_{\cH^{0,r+s-2}} +\|F(u)x\pd_x\pd_t u\|_{0,r+s-2}\,,
\end{align*}
and these are clearly controlled by~$\|u\|_{s,r}$.

Now that we have estimated $\|\bD_{x,\al_j}^{(2)}u^j\|_{\cH^{0,r+s-2}}$,
which gives control over~$\|u\|_{L^\infty_t\cH^{2,r+s-2}}$, 
we can easily obtain bounds for
$\|\bD_{x,\al_j}^{(2)}\pd_t^ku\|_{L^\infty_t\cH^{0,r+s-2-l}}$ by
taking time derivatives in Equation~\eqref{isolate} and repeating the
argument. Estimates for the other terms
$\|\bD_{x,\al_j}^{(2+i)}\pd_t^ku^j\|_{L^\infty_t\cH^{0,r+k}}$ are then
obtained by successively acting with $\bD_{x,\al}^{(i)}$ on
Equation~\eqref{isolate}, with $i=1,2\dots,s-2$. The only difference is
that one has to use
that, by the choice of the range of parameters made in the assumptions~(i)--(iii), the norms~$\|\cdot\|_{\bH^{s',r'}_\al}$ and~$\|\cdot\|_{\cH^{s',r'}}$
are equivalent by Theorem~\ref{T.bH} for all $s'< \frac n2+1$.

\smallskip

\noindent{\em Additional time derivatives and $C^\infty$ estimates.} The proof of the a priori
estimate~\eqref{x-triple} is, in a way, analogous to that
of~\eqref{xtraders}. If we now isolate $\pd_t^2u$ in
Equation~\eqref{isolate}, we find that the component $u^j\in\cV_j^\ga$
satisfies the equation
\begin{multline}\label{isolate-t}
\pd_t^2 u^j = \cG(u)^j-\bD_{x,\al_j}^{(2)}u^j-(\pd_x b^1)\bD_{x,\al_j}u^j
-\pd_i^*(G^{ik}\pd_k u^j)-x\pd_i^*[(b^2)^i\pd_x u^j]\\
-x\pd_x(b^3\pd_t u^j)-\pd_i[(b^4)^i\pd_ t u^j]+\lot
\end{multline}
Multiplying by $x^\rho$, taking $s-1$ derivatives with respect
to~$t$ and using the bound $\triple u_{s,r}<C\de$, we immediately find that
$x^\rho \pd_t^{s+1}u$ satisfies
\[
\|x^\rho \pd_t^{s+1}u\|_{L^\infty_t\cH^{0,r-1}}<C\,.
\]
Likewise, by successively taking $s-2+i$ time derivatives in~\eqref{isolate-t} and repeating the argument, we readily obtain the
bound
\[
\|x^\rho \pd_t^{s+i}u\|_{L^\infty_t\cH^{0,r-i}}<C
\]
for $2\leq i\leq \rho$.

The fact that the solution is smooth in the polyhomogeneous sense if
the initial and boundary data are is a straightforward consequence of
Theorem~\ref{C.Cmr} and
the persistence of regularity principle (see e.g.~\cite{Tao}), which
just means that the time of existence~$T$ does
not depend on the choice of the integer~$s,r$ as long as they are
large enough, so that in this case $\triple{u}_{s,r}$ is finite (although
not uniformly bounded) for all $s,r$ (of course, $\Bga$ is smoothly
polyhomogeneous by construction). This completes the proof of the theorem.
\end{proof}

The statement about the existence of $C^q$ metrics that appears in the
statement of Theorem~\ref{T.main} is an immediate consequence of
Theorem~\ref{T.iteration} due to Corollary~\ref{C.Cmr} provided that
the initial and boundary data are smooth enough. Specifically, by
keeping track of the various choices of exponents that we have made in
the preceding sections we arrive at the following

\begin{corollary}\label{C.numerology}
Given any $q\geq n-1$, let us choose an integer $p>2q+\frac52
n+7$. If $\hg\in C^p$,
$x^2\tg,x^2K\in C^{n-1}\cap C^p\poly$ and they satisfy the constraint
equations and the compatibility
conditions to order~$q$, then there exists a $T>0$ and a unique solution to the
equation $Q(g)=0$ on $(-T,T)\times \bola$ with the above initial and
boundary data, which is of class
$\Bg\in C^{n-1}\cap C^q\poly$.
\end{corollary}

\section{DeTurck's trick revisited}
\label{S.DeTurck}

Corollary~\ref{C.numerology} provides a weakly asymptotically AdS
metric~$g$ that solves the equation $Q(g)=0$ in $(-T,T)\times\bola$,
satisfies the desired initial and boundary conditions.
Our objective in this section is to show that~$g$ is also a
solution of the Einstein equation $\Ric(g)=-ng$, which completes the
proof of Theorem~\ref{T.main}. The standard way of
proving this is via the so-called DeTurck's trick. A textbook presentation of this
method can be found~in~\cite[Chapter 14]{Ringstrom} (see also~\cite{HE}), so we will only
sketch the main ideas and refer to this book for further details. It
should be noticed, however, that the lack of global hyperbolicity and
the fact that the equations that appear are singular at the conformal
boundary ensure that an additional effort is necessary to show that
DeTurck's method actually works in the situation that we are
considering. Fortunately, the estimates that we have derived in the
previous sections of this paper are well suited for this task.

The key idea in DeTurck's method is that, if $Q(g)=0$, the 1-form $W$ introduced
in~\eqref{W} to break the gauge invariance of the Einstein equation
must satisfy the linear hyperbolic equation
\begin{equation}\label{eqW}
\square_g W_\mu+R^\nu_\mu W_\nu=0\,,
\end{equation}
where $R^\nu_\mu:=g^{\nu\la}R_{\mu\la}$ is the tensor obtained by
raising an index of the Ricci tensor of the metric~$g$. When the
metric~$g$ is globally hyperbolic, it is immediate that if $W_\mu=0$
and $\pd_t W_\mu=0$ at $t=0$, then $W\equiv 0$ for all time, which
readily implies that the metric satisfies the Einstein equation
$\Ric(g)=-ng$ because of the structure of the operator~$Q$.

The difficulty here is that Equation~\eqref{eqW} is not globally
hyperbolic. In fact, since~$g$ is weakly asymptotically AdS (which
ensures that $g=x^{-2}\Bg$ for some~$\Bg$ smooth enough up to the
boundary and such that $\Bg^{\mu\nu}x_\mu x_\nu =1$ on $(-T,T)\times\bdry$), a
tedious computation shows that, in~$\cA$, Equation~\eqref{eqW} reads as
\begin{equation}\label{eqW2}
g^{\la\nu}\pd_\la\pd_\nu W_\mu+\frac{(3-n)\, \pd_x W_\mu}x-\frac{n
  W_\mu+(n-1) \Bg^{\la\nu} x_\la W_\nu\, x_\mu}{x^2}+\lot\,,
\end{equation}
where $\lot$ stand for terms with at most one derivative of~$W$ that are smaller at $x=0$ (i.e., they
are of the form  $\cO(1)\, \pd W+\cO(x^{-1})W$). 

Let us now write $W=:W^0+W^3$, with
\[
(W^0)_\mu:= \frac{\Bg^{\la\nu}x_\la W_\nu}{|dx|^2_{\Bg}}x_\mu\,.
\]
This decomposition diagonalizes~\eqref{eqW2} in the sense that the
leading terms of the equation (both in terms of derivatives and
singular behavior at the boundary) are now controlled by scalar operators:
\begin{align*}
\cL_0W_0&:=\bigg(g^{\la\nu}\pd_\la\pd_\nu
+\frac{3-n}x\pd_x-\frac{3n-1}{x^2}\bigg)W_0+\lot\,,\\
\cL_3W_3&:=\bigg(g^{\la\nu}\pd_\la\pd_\nu
+\frac{3-n}x\pd_x-\frac{2n}{x^2}\bigg)W_3+\lot\,,
\end{align*}
where again $\lot$ stands for lower-order terms that are smaller at $x=0$.
Setting $W_j=:x^{\frac n2-1}V_j$ for $j=0,3$, we can now write
\[
\cL_jW_j=:x^{\frac n2-1}\cP_jV_j\,,
\]
where in~$\cA$ the linear operator $\cP_j$ reads as
\begin{multline*}
\cP_jV_j=\bar g^{00}\, \Big(\pd_t^2+
\pd_{\te^i}^*G^{ik}\pd_{\te^k}+
\bD^*_{x,\al_j} b^1\bD_{x,\al_j} + x\, \pd_{\te}^*
\tilde b^2\pd_x\\
+x\pd_x^*\tilde b^3\pd_{\te}+
x\pd_x\tilde b^4\pd_t+x\pd_{\te}\tilde b^5\pd_t\Big)\, V_j\\
+\Big(\tilde b^6 x \pd_xV_j+x \tilde b^7 \pd_tV_j + x \tilde b^8
\pd_{\te}V_j + \tilde b^9 V_j\Big)
\end{multline*}
with $\al_0$ and $\al_3$ defined in Equation~\eqref{alj} 

Since this has the same structure as the operator $\cP_g$ considered
in~\eqref{Pu}, a minor variation of Theorem~\ref{T.linear} proves, in
particular, that any solution $V:=V_0+V_3$ must vanish identically in
$(-T,T)\times\bola$ if it has zero boundary and initial
conditions. The compatibility conditions for the
initial and boundary conditions guarantee that this is indeed the case
(cf.~Appendix~\ref{A.constraint}), so we have proved the following

\begin{theorem}\label{T.DeTurck}
The metric~$g$ constructed in Theorem~\ref{T.iteration} (or Corollary~\ref{C.numerology}) solves the Einstein
equation $\Ric(g)=-ng$ in $(-T,T)\times\bola$.
\end{theorem}

The main result of the paper (Theorem~\ref{T.main}) then follows.

\appendix

\section{Constraint equations and compatibility conditions}
\label{A.constraint}

In this appendix we recall the constraints that
must be satisfied by the initial and boundary data of the Einstein
equations $\Ric(g)+ng=0$. We refer to~\cite{AC} for details.

The initial and boundary conditions are a
Riemannian metric $\tg_{ij}$ on the $n$-dimensional manifold $\bola$, a
second-order tensor $K_{ij}$ on $\bola$ and a Lorentzian metric
$\hg_{\al\be}$ on $\RR\times\bdry$. We also need a function $x$ on
$\obola$, which we assume to be $C^\infty$ up to the boundary. The
connection of these objects with the the Lorentzian Einstein
metric~$g$ on $(-T,T)\times \bola$ is that $\hg_{\al\be}$ is the pullback of $\Bg_{\mu\nu}:=x^2 g_{\mu\nu}$ to
$(-T,T)\times\bdry$, $\tg_{ij}$ is the pullback of~$g_{\mu\nu}$ to the Cauchy surface
$\{0\}\times M$ and $K_{ij}$ is the second fundamental form of the
Cauchy surface in $(-T,T)\times M$ with respect to the
metric~$g_{\mu\nu}$. In terms of regularity, we assume that
$\hg_{\al\be}$ is of class $C^p((-T_,0,T_0)\times\bdry)$, that
$x^2 \tg_{ij}$ is in $C^{n-1}(\obola)\cap C^p\poly(\obola)$ and that $K_{ij}$
can be written as
\[
K_{ij}=\frac1x L_{ij}+\frac1n \cK \tg_{ij}\,,
\]
where $L_{ij}$ is traceless (that is, $\tg^{ij}L_{ij}=0$, so $\cK=\tg^{ij}K_{ij}$) and $L_{ij},\cK\in
C^{n-1}(\obola)\cap C^{q-1}\poly(\obola)$. With some abuse of
notation, throughout this paper we use the shorthand notation
\[
\|x^2K\|_{C^{p-1}_{n-1}}:=\|L_{ij}\|_{C^{p-1}_{n-1}}+\|\cK\|_{C^{p-1}_{n-1}}\,,
\]
and when we say that $x^2K_{ij}$ is in $C^{n-1}\cap C^{p-1}\poly$ we mean that
$L_{ij}, \cK\in C^{n-1}\cap C^{p-1}\poly$. We recall that the
estimates in~\cite{AC} control precisely these quantities (in addition
to $x^2 \tg_{ij}$).

The way to compute $\pd_t^k g_{\mu\nu}|_{t=0}$ from the initial data $(\tg,K)$
is well know, the only difference being that one must take care of the
powers of $x$ that characterize the behavior at infinity of the
metric. An economic way of doing this (see
e.g.~\cite[Section~7.5]{CB} for details) is by embedding $M$ in the
product $(-T_0,T_0)\times M$ and choosing $t\in (-T_0,T_0)$ as a time
coordinate. We can then identity $\{t=0\}$ with $M$ and set, for any
local coordinates on~$M$,
\[
g_{ij}|_{t=0}=\tg_{ij}\,,\qquad g_{ti}|_{t=0}=0\,,\,,\qquad g_{tt}=-x^{-2}\,.
\]
The condition that $K$ be the second fundamental form of
the spatial hypersurface $\{t=0\}$ translates into
\[
\pd_t g_{ij}|_{t=0}=\frac2xK_{ij}
\]
while the time derivatives of the coefficients $g_{t\mu}$ at $0$ are chosen so
as to ensure that the 1-form $W$ (cf.\ Equation~\eqref{W}) vanishes at
$t=0$. Higher order time derivatives of the metric a time~0 can then
be computed from the equation $Q(g)=0$. Because of
Proposition~\ref{P.Qsing} we assume that 
\[
1=\Bg^{\mu\nu}x_\mu x_\nu|_{\{0\}\times\bdry}=\overline{\tg}^{ij}\,
\pd_ix\,\pd_j x|_{\bdry} \,,
\]
where $\overline\tg^{ij}$ is the inverse of
$\overline\tg_{ij}:=x^2g_{ij}$.

The initial data $(\tg_{ij},K_{ij})$ cannot be chosen freely, as the following {\em constraint
  equations}\/ must be satisfied:
\begin{subequations}\label{eq.constraint}
\begin{align}
\widetilde R -K_{ij}K^{ij} + \cK^2&=-n(n-1)\,,\\
\widetilde \nabla^j K_{ji}-\widetilde\nabla_i \cK&=0\,.
\end{align}
\end{subequations}
Here the quantities with tildes are computed using the Riemannian
metric $\tg$, $\widetilde R$ stands for the scalar curvature
of~$\tg$ and indices are raised and lowered using this metric. The proof
goes exactly as in~\cite{Ringstrom}.

This kind of initial data, with the assumption that the objects should
be $C^\infty$ up to the boundary, were considered by Friedrich in his
breakthrough paper~\cite{Friedrich} to construct space-times with
AdS-type behavior at space-like infinity in dimension~4. This has been
discussed in more generality in K\'ann\'ar~\cite{Kannar}. Andersson and Chrusciel~\cite{AC} have established the existence of
many solutions with the right behavior at infinity to the constraint
equations under the additional assumption that $\cK$ is constant, that
it,
\[
\widetilde\nabla_i\cK=0\,.
\]
This extra hypothesis is used to decouple the scalar and vector
constraint equations. These solutions are ``labeled'' by a symmetric
traceless tensor $A^{ij}$ that is sufficiently smooth up to the
boundary (say, in $C^\infty(\obola)$). It is worth mentioning that,
generically, the resulting solutions $(\overline\tg_{ij},L_{ij},\cK)$
are not arbitrarily smooth up to the boundary due to the appearance of
log terms: they are generically in $C^{n-1}\cap C^\infty\poly$,
although there are also ``many'' nontrivial solutions that are smooth up to the
boundary, in which the log terms are absent.

Additionally, one must consider {compatibility conditions} between the initial
conditions $(\tg,K)$ and the boundary datum $\hg$. As is well-known,
solving the Einstein equation in a bounded domain with nontrivial
boundary conditions on the boundary is usually problematic (see
e.g.~\cite{FN} and references therein). Fortunately, in this setting we
can exploit the fact that the metric we want to construct is
asymptotically anti-de Sitter to obtain a manageable set
of compatibility conditions: one have fixed the integers $s,r$ (with
$s+r\leq p$), we only need to impose that the functions
$u_k$, defined in~\eqref{u01} and~\eqref{uk}, belong to $\cH^{1,r}$
for $0\leq k\leq s-1$ and to $\bL$ for $k=s$.  This integrability
condition at infinity is enough to ensure that the arguments in
the paper make sense, essentially because we can integrate by parts in
the proof of Theorem~\ref{T.linear}.
A more intuitive way of understanding this condition is that it is tantamount to saying that
the formal solutions that we calculate at $t=0$
using $(\tg,K)$ (that is, $\pd_t^k\Bg|_{t=0}$ as computed above) and
at $x=0$ using the boundary data (the metrics $\Bga_l$ of
Theorem~\ref{T.peeling} with $l\geq q$) must agree to order~$q$.

\section{Some estimates for the operators $A_\al$ and $A_\al^*$}
\label{A.AA}

The integral operators $A_\al$ and $A_\al^*$, defined
in~\eqref{defAA*}, play a key role in some arguments presented in
Sections~\ref{S.Sobolev} and~\ref{S.nonlinear}. Therefore we will
record here some estimates the we proved in~\cite[Theorem~3.1 and
Proposition~3.3]{JMPA}, where as usual we assume that $\al>1$. For the
benefit of the reader, we also include a sketch of the proof.

\begin{theorem}[\cite{JMPA}]\label{T.AA*}
The following statements hold:
\begin{enumerate}
\item Acting on one-variable functions, the operators $A_\al$ and $A_\al^*$ define continuous maps
\[
\bLx\to L^\infty_x\,.
\] 

\item The operators $\frac1x A_\al$ and $\frac1xA_\al^*$ are continuous maps
\[
\bLx\to \bLx \qquad \text{and} \qquad \bL\to\bL\,.
\]

\item If $ u$ is a function in $\bL(\cA)$ with $\bD_{x,\al} u$ in
  $\bL(\cA)$, then
\[
 u(x,\te)=(A_\al \bD_{x,\al} u)(x,\te)\,.
\]

\item If $ u$ is a function in $\bL(\cA)$ with $\bD_{x,\al}^* u$ in
  $\bL(\cA)$, then
\[
 u(x,\te)=(A_\al^*\bD_{x,\al}^*  u)(x,\te)+f(\te)\, x^{\al-1}\,,
\]
the function $f(\te)$ being bounded in $L^2_\te\equiv L^2(\bdry)$ by  
\[
\|f\|_{L^2_\te}\leq C (\| u\|_{\bL}+ \|\bD_{x,\al}u\|_\bL)\,.
\]
\end{enumerate}
\end{theorem}

\begin{proof}
  We can assume that $u$ is smooth and supported in the region
  $0<x<1$.  Let us begin analyzing the mapping properties of
  $A_\al^*$. 

  In view of the expression for $A^*_\al$, we will use the Hardy
  inequality
\begin{equation}\label{Hardy}
\int_0^1x^{2\al-2r-1} \bigg(\int_x^1y^{1-\al} \vp(y)\,
dy\bigg)^2dx\leq C\int_0^1 x^{3-2r} \vp(x)^2\, dx\,,
\end{equation}
with $r=0,1$.
To prove this, let us set
\[
 \psi(x):=\int_x^1y^{1-\al} \vp(y)\,dy\,.
\]
Then integrating by parts and using the Cauchy--Schwarz inequality we
find
\begin{align*}
\int_0^1x^{2\al-2r-1} \psi^2\, dx&
=\frac1{\al-1}\int_0^1  \vp\, \psi\, x^{\al-2r-1}\,
dx\\
&=\frac1{\al-r}\int_0^1  (x^{\al-r-\frac12} \psi)\,
(x^{\frac32-r} \vp)\, dx\\
&\leq \frac{1}{\al-r}\,\bigg(\int_0^1x^{2\al-2r-1} \psi^2\, dx\bigg)^{\frac12}\, \bigg(\int_0^1 x^{3-2r} \vp^2\, dx\bigg)^{\frac12}\,.
\end{align*}
This proves~\eqref{Hardy}. This implies that, with $r=0,1$, $\frac1xA_\al^*$
is a bounded map
\[
L^2((0,1),x^{3-2r} dx)\to L^2((0,1), x^{1-2r}\, dx)\,,
\]
and with $r=1$ this implies that $\frac1x A_\al^*:\bLx\to\bLx$. Since the star
denotes the adjoint with respect to the $\bLx$~product, a standard
duality argument then ensures that $A_\al$ is a bounded map
\[
L^2((0,1),x^{1+2r}dx)\to L^2((0,1), x^{2r-1}dx)\,,
\]
which with $r=0$ implies that $\frac 1xA_\al:\bLx\to\bLx$. The fact
that this also corresponds to $\bL\to\bL$ bounds is immediate.

Let us now pass to the pointwise bounds. To prove~(i) for $A_\al^*$ we
utilize the Cauchy-Schwarz inequality to write
\begin{align*}
  \big|A_\al^* \vp(x)\big|&=x^{\al-1}\bigg|\int_x^1y^{1-\al} \vp(y)\,
  dy\bigg|\\
  & \leq x^{\al-1}\bigg(\int_x^1y^{1-2\al}\,
  dy\bigg)^{\frac12}\bigg(\int_x^1y\,  \vp(y)^2\, dy\bigg)^{\frac12}\\
&\leq \| \vp\|_{\bLx}\bigg(\frac{1-x^{\al-1}}{2-2\al}\bigg)^{1/2}\\
&\leq (2-2\al)^{-\frac12} \| \vp\|_{\bLx}\,.
\end{align*}
The $L^\infty_x$ estimate for $A_\al$ is similar.

To prove~(iv), notice that if $u_1:= \bD_{x,\al}^* u\in \bL$, we can
solve the ODE
\[
\bD_{x,\al}^* u=u_1
\]
to write
\[
u=A_\al^*(u_1)+ f(\te)\, x^{\al-1}
\]
for some function~$f(\te)$. Moreover,
\[
\|f\|_{L^2_\te}= C\|f(\te)\, x^{\al-1}\|_\bL\leq
C(\|u\|_\bL+\|A_\al^*(u_1)\|_\bL)\leq  C(\|u\|_\bL+\|u_1\|_\bL)\,,
\]
where we have used that $A_\al^*:\bL\to\bL$ by~(ii). To prove~(iii),
the reasoning is analogous:
again we can solve the ODE
\[
\bD_{x,\al} u=u_2
\]
to write
\[
u=A_\al(u_2)+ f_2(\te)\, x^{-\al}\,,
\]
but we infer that $f_2$ must be $0$ because $x^{-\al}$ is not in
$\bLx$. The theorem then follows.
\end{proof}

\section*{Acknowledgements}

We would like to thank the referee for his detailed comments and suggestions, which helped to significantly improve our paper. A.E.\ is supported by the ERC Starting Grant 633152 and thanks McGill University for hospitality
and support. A.E.'s research is supported in part by the ICMAT Severo
Ochoa grant SEV-2015-0554. The research of N.K.\ is supported by NSERC grant
RGPIN 105490-2011.

\bibliographystyle{amsplain}

\end{document}